\documentclass{siamart190516}
\usepackage[a4paper, margin=3cm]{geometry}

\usepackage{graphicx} 
\usepackage{amsfonts,verbatim,cite,subcaption,tikz,float,multicol,url,booktabs,mathabx}
\usetikzlibrary{positioning}

\usepackage[section]{placeins}
\usepackage[algo2e,ruled,linesnumbered]{algorithm2e}

\newcommand{\argmin}{\mathop{\arg\min}}

\newcommand{\Ttran}{\mathsf{T}}

\newcommand{\Htran}{\mathsf{H}}

\newcommand{\defi}{:=}

\newcommand{\normml}{\left\vert\kern-0.25ex\left\vert\kern-0.25ex\left\vert}  
\newcommand{\normmr}{\right\vert\kern-0.25ex\right\vert\kern-0.25ex\right\vert}
\newcommand{\order}{\mathcal{O}}

\newcommand{\kryl}{\mathcal{K}}		% krylov subspace

\newcommand{\region}{\mathcal{D}}

\newcommand{\R}{\mathbb{R}}
\newcommand{\C}{\mathbb{C}}

\newcommand{\range}{\mathsf{range}}

\newcommand{\old}{(\mathrm{old})}
\newcommand{\new}{(\mathrm{new})}
\newcommand{\mpr}{\mathbf{u}}
\newcommand{\fig}{eps}

\newcommand{\figsizeD}{0.45\textwidth}
\newcommand{\figsizeT}{0.3\textwidth}
\newcommand{\figsizeQ}{0.22\textwidth}

\providecommand{\spa}[1]{\mathsf{span}\{#1\}}

\providecommand{\abs}[1]{\lvert#1\rvert}
\providecommand{\norm}[1]{\lVert#1\rVert}

\newtheorem{remark}[theorem]{Remark}
\crefname{algocf}{Algorithm}{Algorithms}

\numberwithin{equation}{section}

\begin{document}
\title{Restoring similarity in randomized Krylov methods with applications to eigenvalue problems and matrix functions
% \thanks{\today.}
}
\author{Laura Grigori\thanks{PSI Center for Scientific Computing, Theory and Data, 5232 Villigen, Switzerland, and Institute of Mathematics, EPF Lausanne, 1015 Lausanne, Switzerland (\href{mailto:laura.grigori@epfl.ch}{laura.grigori@epfl.ch})}
\and  Daniel Kressner\thanks{Institute of Mathematics, EPF Lausanne, 1015 Lausanne, Switzerland (\href{mailto:daniel.kressner@epfl.ch}{daniel.kressner@epfl.ch} and \href{mailto:nian.shao@epfl.ch}{nian.shao@epfl.ch} and \href{mailto:nian.shao@epfl.ch}{igor.simunec@epfl.ch})}
\and Nian Shao\footnotemark[2]
\and Igor Simunec\footnotemark[2]}

    \headers{Similarity-restoring randomized Arnoldi}{Laura Grigori, Daniel Kressner, Nian Shao, Igor Simunec}
\maketitle

\begin{abstract}
	The randomized Arnoldi process has been used in large-scale scientific computing because it produces a well-conditioned basis for the Krylov subspace more quickly than the standard Arnoldi process. 
	However, the resulting Hessenberg matrix is generally not similar to the one produced by the standard Arnoldi process, which can lead to delays or spike-like irregularities in convergence.
	In this paper, we introduce a modification of the randomized Arnoldi process that restores similarity with the Hessenberg matrix generated by the standard Arnoldi process. 
	This is accomplished by enforcing orthogonality between the last Arnoldi vector and the previously generated subspace, which requires solving only one additional least-squares problem.
	When applied to eigenvalue problems and matrix function evaluations, the modified randomized Arnoldi process produces approximations that are identical to those obtained with the standard Arnoldi process. 
	Numerical experiments demonstrate that our approach is as fast as the randomized Arnoldi process and as robust as the standard Arnoldi process.
\end{abstract}

\section{Introduction}

Krylov subspace methods are among the most effective approaches for solving large-scale problems, ranging from the solution of linear systems, to the computation of eigenvalues and evaluation of matrix functions. 
Given a square matrix $A \in \C^{n\times n}$ and a nonzero initial vector $b \in \C^{n}$, the $m$th Krylov subspace associated with $A$ and $b$ is generated by repeatedly multiplying $A$ with $b$: 
\begin{equation*} 
    % \label{eq:krylovsubspace}
    \mathcal{K}_{m}(A,b) \defi \spa{b,Ab,\dotsc,A^{m-1}b}.
\end{equation*} 
The Arnoldi process~\cite[Sec. 10.5]{Golub2013} is used to construct an orthonormal basis $Q_{m}$ for $\mathcal{K}_{m}(A,b)$, whose main costs are $m$ matrix-vector products with $A$, and the orthogonalization of $Q_{m}$ via a Gram--Schmidt process with complexity $\order(nm^2)$. 
In particular, orthogonalization may dominate the computational cost when $m$ becomes (moderately) large. Several strategies have been considered in the literature to reduce this cost, including incomplete orthogonalization \cite{Saad1980} and restarting \cite{sorensen1992implicit}.

Recently, randomized orthogonalization has been proposed to reduce the cost of orthogonalization in the Arnoldi process, for instance with a randomized Gram--Schmidt method \cite{Balabanov2022,BalabanovGrigori25block} or with a randomized Householder QR decomposition~\cite{GrigoriTimsit24}. 
Given a (low-dimensional) subspace $\mathcal{W}\subset\C^{n}$, the randomized Gram--Schmidt method employs a subspace embedding $\Omega \in \C^{d \times n}$ to sketch vectors in $\mathcal{W}$ to $\C^d$ with $d \ll n$, while approximately preserving their norms.
Compared to standard Gram--Schmidt, the computation and communication costs are reduced because inner products are formed with sketched vectors of length $d$ rather than with original length-$n$ vectors.
Instead of an orthonormal basis $Q_{m}$, randomized Gram--Schmidt generates an $\Omega$-orthonormal basis $U_{m}$, that is, the columns of $\Omega U_{m}$ form an orthonormal basis. 
When using high-quality random subspace embeddings, such as the ones surveyed in \cite[Sec.~2]{murray2023randomized}, this ensures --- with high probability --- that $U_m$ is a well-conditioned basis.
During the last few years, the resulting randomized Arnoldi process has been employed for the solution of various large-scale
problems~\cite{TGB23,PSS25,GuettelSchweitzer23,de2024randomized,de2025randomized,grigori2025randomized,CKN24}. For a recent overview, we refer to \cite{DGST25}.

While its performance gains make randomized Arnoldi an attractive choice, the transition from orthonormality to $\Omega$-orthonormality incurs certain nuisances. An important feature of standard Arnoldi is that it not only returns an orthonormal basis $Q_m$ but also the Hessenberg matrix $G_m  = Q_m^\Htran A Q_m$. The latter plays a crucial role in algorithms; for example, the eigenvalues of $G_m$ are the so called Ritz values that approximate eigenvalues of $A$. Randomized Arnoldi also returns such an $m\times m$ Hessenberg matrix, but this matrix is, in general, \emph{not} similar to $G_m$ and, in particular, its eigenvalues are different.
This discrepancy can lead to both numerical and theoretical difficulties. For example, one may obtain complex eigenvalue approximations even when $A$ is Hermitian~\cite{de2025randomized}. More subtle issues arise in the context of linear systems and matrix functions, where convergence may be delayed~\cite{GuettelSchweitzer23} or exhibit sporadic spikes~\cite{TGB23}. 
Moreover, from a theoretical viewpoint, establishing convergence results becomes highly challenging \cite{PSS25} because the numerical range of the Hessenberg matrix is difficult to control, as it is no longer contained in the numerical range of $A$. 
All of these issues are caused by the lack of similarity between the Hessenberg matrices and can be resolved by restoring this property.

In this work, we propose a modification of randomized Arnoldi that constructs a Hessenberg matrix similar to the one generated by standard Arnoldi. The key idea is to make all basis vectors $\Omega$-orthonormal \emph{except} for the last one, which is instead orthogonal to all previous vectors. Such a basis can be obtained by performing randomized Gram--Schmidt, followed by the solution of a least-squares problem that enforces orthogonality of the last vector. This general framework enables us to recover, in exact arithmetic, the same approximate solutions as the ones obtained through standard Arnoldi, while still exploiting much of the efficiency of randomized Gram--Schmidt. We demonstrate this approach by applying it to the randomized Krylov--Schur algorithm for eigenvalue problems and the randomized Arnoldi method for computing the action of matrix functions $f(A)b$.
For eigenvalue problems, in addition to establishing its equivalence with standard Krylov--Schur in exact arithmetic, we also prove a backward stability result analogous to that of the standard Krylov--Schur method in \cite{stewart2002krylov}.
Let us remark that the application to matrix functions is closely related to one of the approaches considered in \cite{CKN24}; see Remark~\ref{rem:matfun-comparison-with-ckn} below for more details.

The remainder of the manuscript is organized as follows. In \cref{sec:algorithm-description} we review the standard and randomized Arnoldi process and introduce our similarity-restoring correction to randomized Arnoldi. In \cref{sec:eigenvalue} we apply the proposed framework to the randomized Krylov--Schur algorithm for eigenvalue problems, and in \cref{sec:matrix-functions} we employ it for the computation of matrix functions. Numerical comparisons with both standard and randomized algorithms are presented in \cref{sec:numerical-experiments}.

\section{A new variant of the randomized Arnoldi process}
\label{sec:algorithm-description}
\subsection{Arnoldi and randomized Arnoldi process}
\label{subsec:arnoldi-process}
\paragraph{Arnoldi process}

We start by briefly recalling the standard Arnoldi process, which first normalizes the initial vector by setting $q_1 = b/\|b\|$ and then computes, for $k =1,2,\ldots$, the basis vector $q_{k+1}$ by performing a matrix-vector product with $A$ and one step of the (classical) Gram--Schmidt process:
\begin{equation*}
    w = A q_{k},\quad g_{k}=Q_{k}^{\Htran}w,\quad \widetilde{q}_{k+1} = w-Q_{k}g_{k},\quad g_{k+1,k} = \norm{\widetilde{q}_{k+1}},\quad q_{k+1} = \widetilde{q}_{k+1}/g_{k+1,k}.
\end{equation*}
Here and in the following, we will ignore the rare (and usually fortunate) event that the Arnoldi process breaks down because of $\widetilde{q}_{k+1} = 0$. 
Performing $m$ steps, collecting the vectors $q_k$ in an orthonormal matrix $Q_{m}$ and the coefficients $g_k,g_{k+1,k}$ in an upper Hessenberg matrix $G_{m}$, we obtain the Arnoldi decomposition
\begin{equation}
    \label{eq:AD}
    AQ_{m} = Q_{m}G_{m}+g_{m+1,m}q_{m+1}e_{m}^{\Htran},
\end{equation}
where $e_{m} \in \C^m$ denotes the $m$th unit vector.
When $A$ is Hermitian, the Arnoldi process reduces to the Lanczos process, and the Hessenberg matrix $G_{m}$ becomes a tridiagonal Hermitian matrix.

\paragraph{Randomized Arnoldi process}

As mentioned in the introduction, orthogonalization can dominate the computational cost in the Arnoldi process. To accelerate this part, one can replace the Gram--Schmidt process by the randomized Gram--Schmidt process \cite{Balabanov2022}.
To describe this process, we need the following definition.

\begin{definition}
	\label{defn:epsilon-subspace-embedding}
	Given $\varepsilon \in (0,1)$, a matrix $\Omega \in \C^{d \times n}$ is called an $\varepsilon$-subspace embedding for a subspace $\mathcal{W} \subset \C^{n}$ if 
	\begin{equation*}
		(1-\varepsilon)\norm{w}^2 \le \norm{\Omega w}^2 \le (1+\varepsilon)\norm{w}^2, \quad \forall w \in \mathcal W.
	\end{equation*}     
\end{definition}
The matrix $\Omega$ from \Cref{defn:epsilon-subspace-embedding} is usually called sketching matrix, and it has the benefit of reducing the length of a vector from $n$ to $d$, while approximately preserving its norm. Sketching makes sense when the dimension $m$ of the subspace  $\mathcal{W}$ is much smaller than $n$.
In the context of randomized Arnoldi, $\mathcal{W}$ refers to the Krylov subspace $\mathcal{K}_{m+1}(A,b)$, which is actually of dimension $m+1$.

Often, the subspace $\mathcal{W}$ is unknown, or it may change during the execution of the algorithm.
In these situations, it is convenient to use an \emph{oblivious $\varepsilon$-subspace embedding}, that is, a sketching matrix $\Omega$ that is (necessarily) random and represents a $\varepsilon$-subspace embedding for an \emph{arbitrary} subspace of dimension $m$ with probability $1-\delta$ for some $\delta \in (0,1)$. Several different ways to construct such oblivious subspace embeddings have been proposed in the literature; see, e.g., \cite{MartinssonTropp20} for an overview. A Gaussian random sketching matrix $\Omega$ comes with strong theoretical guarantees: A sketching dimension $d = \mathcal{O}(\varepsilon^{-2}(m + \log(1/\delta)))$ suffices to obtain an oblivious $\varepsilon$-subspace embedding~\cite[Thm.~2.3]{Woodruff14}. 
However, in many scenarios, it is preferable to use sketching matrices that admit fast matrix-vector products. In this work, we construct $\Omega \in \R^{d \times n}$ as a sparse sign matrix (see, e.g.,\cite[Sec.~9.2] {MartinssonTropp20}): Each column of $\Omega$ is an i.i.d.~random vector, with $\xi$ random sign entries located at random positions. It was shown in \cite[Cor.~2.2]{CDD25} that $\Omega$ is an oblivious $\varepsilon$-subspace embedding provided that $d = \mathcal{O}(\varepsilon^{-2}(m + \log(m/\delta)))$ and $\xi = \mathcal{O}(\varepsilon^{-1}\log(m/\delta)^{5/2} + \log(m/\delta)^4)$. 
In practice, simply setting $\xi = 8$ has been observed to yield good results~\cite{TYUC19}, which will be used in our experiments as well.

Given an $\varepsilon$-subspace embedding $\Omega\in\C^{d\times n}$ for $\mathcal{W}\subset\C^{n}$, where $d\ll n$, we say that two vectors $w_{1}$ and $w_{2}$ in $\mathcal{W}$ are $\Omega$-orthogonal if $\Omega w_{1}$ and $\Omega w_{2}$ are orthogonal in $\C^{d}$.
Given a basis $W$ (not necessarily $\Omega$-orthonormal), a new vector $w$ is $\Omega$-orthogonalized with respect to $W$ by carrying out one step of randomized Gram--Schmidt:
\begin{equation} \label{eq:randgs}
     w - W(\Omega W)^{\dagger}\Omega w,
\end{equation}
where $(\cdot)^\dagger$ denotes the Moore--Penrose pseudoinverse of a matrix.

Replacing the Gram--Schmidt process in Arnoldi with its randomized variant~\eqref{eq:randgs} leads to the randomized Arnoldi process, which first $\Omega$-normalizes the initial vector 
by setting $u_1 = b/\|\Omega b\|$ and then computes, for $k =1,2,\ldots$, the basis vector $u_{k+1}$ as follows:
\begin{equation*}
    w = Au_{k},\quad h_{k}=(\Omega U_{k})^{\dagger}(\Omega w),\quad \widetilde{u}_{k+1} = w-U_{k}h_{k},\quad h_{k+1,k} = \norm{\Omega \widetilde{u}_{k+1}},\quad u_{k+1} = \frac{\widetilde{u}_{k+1}}{h_{k+1,k}}.
\end{equation*}
Performing $m$ steps of this process generates a so-called \emph{randomized Arnoldi decomposition}~\cite{Balabanov2022}:
\begin{equation}
    \label{eq:rAD}
    AU_{m}=U_{m}H_{m}+h_{m+1,m}u_{m+1}e_{m}^{\Htran},
\end{equation}
where $[U_{m},u_{m+1}]$ contains $\Omega$-orthonormal columns.
The $m\times m$ matrix $H_{m}$ is upper Hessenberg and collects the quantities
$h_k$  and $h_{k+1,k}$, just as in the standard Arnoldi process.
It is worth noting that, in contrast to standard Arnoldi, the matrix 
$H_{m}$ is, in general, neither Hermitian nor tridiagonal when $A$ is Hermitian.

\subsection{Restoring similarity}
\label{subsec:newdecomp}

Suppose that we start the Arnoldi process and the randomized Arnoldi process with the same (nonzero) vector $b$, that is, $q_{1}$ in~\cref{eq:AD} and  $u_{1}$ in~\cref{eq:rAD} are collinear. 
As already discussed in the introduction, the corresponding Hessenberg matrices $G_{m}$ and $H_{m}$ are, in general, not similar, which incurs theoretical and practical issues. In the following, we propose a strategy to modify the randomized Arnoldi process and restore the similarity of the Hessenberg matrix.

Given $A\in\C^{n\times n}$, consider the following \emph{Krylov decomposition}:
\begin{equation}
    \label{eq:OKD}
    AU_{m} = U_{m}\widehat{H}_{m}+\widehat{u}_{m+1}c_{m}^{\Htran} 
    \quad \text{and}\quad 
    U_{m}^{\Htran}\widehat{u}_{m+1}=0,
\end{equation}
where $U_{m}\in\C^{n\times m}$ is assumed to be a full-rank matrix (\emph{not} necessarily orthonormal or $\Omega$-orthonormal) and $c_{m}\in\C^{m}$ is a general vector (not necessarily a scalar multiple of $e_{m}$).
Obviously, the Arnoldi decomposition~\cref{eq:AD} takes the form~\cref{eq:OKD}. While 
the randomized Arnoldi decomposition \cref{eq:rAD} takes the form of the more general Krylov decomposition introduced by Stewart~\cite{stewart2002krylov}, it does not take the form~\eqref{eq:OKD} because the orthogonality condition $U_{m}^{\Htran}\widehat{u}_{m+1}=0$ is not satisfied.

A Krylov decomposition of the form~\eqref{eq:OKD} implies a similarity relation between $\widehat{H}_{m}$ and the Hessenberg matrix generated by the standard Arnoldi process.

\begin{theorem}
    \label{thm:OKD}
    Given the Krylov decomposition \cref{eq:OKD}, let $\widehat{Q}_{m}$ be an orthonormal basis of $\range(U_{m})$. Then $\widehat{H}_{m}$ and $\widehat{G}_{m} := \widehat{Q}_{m}^{\Htran}A\widehat{Q}_{m}$ are similar. Moreover, the following Krylov decomposition holds:
    \begin{equation}
        \label{eq:aKD}
        A\widehat{Q}_{m}=\widehat{Q}_{m}\widehat{G}_{m}+\widehat{u}_{m+1}\widehat{c}_{m}^{\Htran},
        \quad\text{where}\quad \widehat{c}_{m} = (\widehat{Q}_{m}^{\Htran}U_{m})^{-\Htran}c_{m}, 
        \quad \text{and}\quad      \widehat{Q}_{m}^{\Htran}\widehat{u}_{m+1}=0.
    \end{equation}
\end{theorem}
The relation \cref{eq:aKD} plays an important role in our further analysis. In the following, we refer to~\cref{eq:aKD} as \emph{orthonormal Krylov decomposition associated with~\cref{eq:OKD}}. 
\begin{proof}[Proof of \cref{thm:OKD}]
Because $\widehat{Q}_{m}$ is an orthonormal basis of $\range(U_{m})$, the matrix $R:=\widehat{Q}_{m}^{\Htran}U_{m}$ is nonsingular and $U_{m}=\widehat{Q}_{m}\widehat{Q}_{m}^{\Htran}U_{m} = \widehat{Q}_{m} R$ holds. Multiplying both sides of~\eqref{eq:OKD} with $R^{-1}$ gives
\[
A \widehat{Q}_{m} = U_m  \widehat{H}_{m} R^{-1} +\widehat{u}_{m+1} {c}_{m}^{\Htran}R^{-1}  = \widehat{Q}_m R  \widehat{H}_{m} R^{-1} +\widehat{u}_{m+1} \widehat{c}_{m}^{\Htran}.
\]
Setting $\widehat{G}_m:= R  \widehat{H}_{m} R^{-1}$ completes the proof.
\end{proof}

\noindent Let us point out the following two important properties of the Krylov decomposition~\cref{eq:OKD}:
\begin{enumerate}
    \item The residual %of the decomposition 
    is orthogonal to the subspace: $AU_{m}-U_{m}\widehat{H}_{m}\perp \range(U_{m})$;
    \item $\widehat{H}_m$ is similar to a matrix $\widehat{G}_m$ whose numerical range satisfies $W(\widehat{G}_m)\subset W(A)$, where we recall that
	$
		W(A) \defi \bigl\{ w^{\Htran} Aw \mid \|w\| = 1  \bigr\}.
	$
	In particular, the eigenvalues of $\widehat{H}_m$ are always contained in $W(A)$, although, in general, $W(\widehat H_m) \not\subset W(A)$. Among others, this implies that $\widehat{H}_m$ has real eigenvalues when $A$ is a Hermitian.

\end{enumerate}

\cref{thm:OKD} also works the other way. Suppose one has an Arnoldi decomposition~\cref{eq:AD} with orthonormal basis $Q_{m}$ and Hessenberg matrix $G_m$. This clearly takes the form~\eqref{eq:aKD}. Now, if one has another decomposition~\cref{eq:OKD} such that
$\range(Q_{m})=\range(U_{m})$ then
\cref{thm:OKD} implies that ensuring the orthogonality condition $U_m^\Htran\widehat{u}_{m+1} = 0$ ensures the similarity of 
$\widehat{H}_{m}$ and $G_{m}$.

\paragraph{Implications for linear systems}
Suppose that we want to solve a linear system $Ax = b$ with the full orthogonalization method (FOM), which was proposed for the general non-Hermitian case in~\cite{Saad81} and further analyzed in \cite{EiermannErnst01}. The defining property of the  approximate solution $x_m \in \kryl_m(A, b)$ computed by FOM is that its residual satisfies the Galerkin condition: $A x_m - b \perp \kryl_m(A, b)$. When $A$ is Hermitian positive definite, FOM reduces to the classical conjugate gradient (CG) method~\cite{HestenesStiefel52}, and $x_m$ minimizes the $A$-norm of the error over the Krylov subspace.
Given the Arnoldi decomposition \cref{eq:AD} with $q_1 = b/\norm{b}$, the approximate solution computed by FOM is given by $x_m = Q_m G_m^{-1} \norm{b} e_1$, where $e_1 \in \C^m$ denotes the first unit vector.

Randomized FOM~\cite{TGB23} employs a randomized Arnoldi process and imposes the sketched Galerkin condition $\Omega(A x_m - b) \perp \Omega \kryl_m(A, b)$ on the residual. The resulting approximation differs from FOM and, indeed, the convergence of randomized FOM may exhibit irregularities and spikes: see, for instance, the numerical experiments in \cite[Sec.~5]{TGB23}.

On the other hand, \emph{if} we had a Krylov decomposition~\cref{eq:OKD} with $u_1 = b/\beta$ for some $\beta \in \R$, we could cheaply recover the FOM approximation even if $U_m$ is not orthonormal. Indeed, letting $x_m = U_m y_m$ and imposing $A x_m - b \perp \range(U_m)$ gives us
\begin{equation*}
	0 = U_m^\dagger A U_m y_m - \beta U_m^\dagger U_m e_1 = \widehat{H}_m y_m - \beta e_1 \quad \implies \quad x_m = \beta U_m \widehat{H}_m^{-1} e_1,
\end{equation*}
where we used $U_m^\dagger A U_m = \widehat{H}_m$, which follows from \cref{eq:OKD}.

We will show in \cref{sec:eigenvalue} and \cref{sec:matrix-functions} that analogous conclusions apply to eigenvalue computations and matrix function evaluations, respectively.

\subsection{Similarity-restoring randomized Arnoldi}
\label{subsec:newrAD}

In view of the discussion above, it remains to relate randomized Arnoldi decompositions to a Krylov decomposition of the form~\cref{eq:OKD}. Given a randomized Arnoldi decomposition~\cref{eq:rAD}, we
solve the linear least-squares problem 
\begin{equation}
	\label{eqn:correction-least-squares-problem}
    \widehat{h}_{m}=  \argmin_{h \in \C^m} \norm{U_m h - u_{m+1}}=U_{m}^{\dagger}u_{m+1} = \big( U_m^\Htran U_m\big)^{-1}U_m^\Htran u_{m+1}.
\end{equation}
We then rearrange the randomized Arnoldi decomposition \cref{eq:rAD} as follows:
\begin{equation}
    \label{eq:rKd2oKD}
    \begin{aligned}
        AU_{m}&=U_{m}H_{m}+h_{m+1,m}u_{m+1}e_{m}^{\Htran} \\ 
        &=U_{m}(H_{m}+h_{m+1,m}\widehat{h}_{m}e_{m}^{\Htran})+h_{m+1,m}(u_{m+1}-U_{m}\widehat{h}_{m})e_{m}^{\Htran}\\ 
        &=U_{m}\widehat{H}_{m}+\widehat{u}_{m+1}c_{m}^{\Htran},
    \end{aligned}
\end{equation}
where we set
\begin{equation}
 \label{eq:updatearnoldi}\widehat{u}_{m+1} := u_{m+1}-U_{m}\widehat{h}_{m},\quad c_m := h_{m+1,m} e_m,\quad \widehat{H}_{m} := H_{m}+ \widehat{h}_{m}c_{m}^{\Htran}.
\end{equation}
By definition~\eqref{eqn:correction-least-squares-problem} of $\widehat{h}_{m}$, it follows that $U_m^\Htran\widehat{u}_{m+1} = 0$. In particular, the Krylov decomposition~\eqref{eq:rKd2oKD} takes the form~\cref{eq:OKD}.

We call the process above \emph{similarity-restoring randomized Arnoldi}, summarized in \cref{algo:RAC}. 
\begin{algorithm2e}[htbp]
	\caption{Similarity-restoring randomized Arnoldi} \label{algo:RAC}
    Obtain randomized Arnoldi decomposition~\cref{eq:rAD}\;
    Solve linear least-squares problem~\cref{eqn:correction-least-squares-problem}\;
    Construct Krylov decomposition~\cref{eq:rKd2oKD} by computing the update~\cref{eq:updatearnoldi}\;
\end{algorithm2e}

A common way to solve the least-squares problem \cref{eqn:correction-least-squares-problem} is to
form the Gram matrix $U_m^\Htran U_m$, compute its Cholesky decomposition $U_m^\Htran U_m = L L^\Htran$, and then solve
\begin{equation*}
	\widehat{h}_m = (U_m^\Htran U_m)^{-1} U_m^\Htran u_{m+1} = L^{-\Htran} L^{-1} U_m^\Htran u_{m+1}.
\end{equation*}
In the following cost analysis, flop counts are reported assuming real-valued arithmetic for simplicity.
The most costly part of this approach is the computation of the Gram matrix $U_{m}^{\Htran}U_{m}$, which requires $m^{2}n$ flops, provided that one exploits symmetry and only computes the upper triangular part of $U_{m}^{\Htran}U_{m}$.\footnote{This is done, for instance, by the BLAS routine \texttt{syrk}, and appears to be automatically carried out when performing matrix-matrix multiplication in MATLAB.}
The $\Omega$-orthogonalization in the randomized Arnoldi process costs another $m^{2}n$ flops for updating the columns of $U_m$, plus lower-order terms for computing the sketched vectors and the inner products between sketched vectors. Thus, in addition to the matrix-vector products with $A$, the leading term in the computational cost of \cref{algo:RAC} is $2 m^{2}n$ flops.  

The orthogonalization of $Q_m$ using classical/modified Gram--Schmidt (CGS/MGS) within the standard Arnoldi process also costs $2 m^{2}n$ flops. Thus, \cref{algo:RAC} requires the same amount of flops as the standard Arnoldi process and twice as much compared to randomized Arnoldi without the similarity-restoring correction.  However, \cref{algo:RAC} significantly benefits from the fact that half of its flops related to orthogonalization are in the computation of $U_m^\Htran U_m$, which is a BLAS-3 operation and can take full advantage of modern computational hardware. Indeed, in \cref{sec:numerical-experiments} we will see that, in practice, correction only yields a minor increase to the computational time of randomized Arnoldi.
Note that we expect the basis $U_m$ to be well-conditioned and, hence, the solution of the least-squares problem via Cholesky is numerically stable. 

An alternative approach to solving the least-squares problem \cref{eqn:correction-least-squares-problem} is to use an iterative method such as LSQR \cite{PaigeSaunders82}. Because $U_m$ is well-conditioned, there is no need to use a preconditioner.
Each iteration of LSQR involves one matrix-vector product with $U_m$ and one with $U_m^\Htran$, at a cost of approximately $4 mn$ flops per iteration. The total cost for $\ell$ iterations of LSQR is thus about $4 \ell mn$ flops, which implies that we can only expect LSQR to be more efficient than the Cholesky approach if LSQR converges in $\ell < m/4$ iterations. In practice, since LSQR is inherently sequential and the computation of the Gram matrix exploits BLAS-3 operations, we anticipate that $\ell$ actually needs to be significantly smaller than $m/4$ in order to observe a practical speedup by using LSQR instead of the Cholesky approach. The numerical comparisons in \cref{sec:numerical-experiments} support this conclusion and demonstrate that Cholesky is usually the more efficient approach, except when the subspace dimension $m$ is rather large and/or low accuracy for the solution of \cref{eqn:correction-least-squares-problem} is sufficient.

In principle, the similarity-restoring correction~\cref{eq:rKd2oKD} to the randomized Arnoldi decomposition~\cref{eq:rAD} could also be employed when the latter is obtained implicitly through \emph{whitening} a basis $W_m$ computed via a cheap partial orthogonalization method, such as the $k$-truncated Arnoldi algorithm; 
see, e.g., \cite[Sec.~3]{PSS25} for more details. 
In this work, we do not consider this variant: although for simple problems it may work effectively, for more difficult ones it would be nontrivial to implement the solution of the least-squares problem \cref{eqn:correction-least-squares-problem} robustly and efficiently. Indeed, solving~\cref{eqn:correction-least-squares-problem} accurately is challenging because of the ill-conditioning of $W_m$, and in turn computing such a solution would increase the computational complexity of whitening. Addressing these aspects would require further investigation and is beyond the scope of this work.

\section{Application to eigenvalue problems}
\label{sec:eigenvalue}

In this section, we apply the similarity-restoring method described in \cref{algo:RAC} in the context of the Arnoldi method for computing a few exterior eigenvalues and the corresponding eigenvectors of a large-scale matrix $A$.
\subsection{Arnoldi, Krylov--Schur and randomized Krylov--Schur method}
\label{subsec:krylov-schur-method}

\paragraph{Basic Arnoldi method} Given an Arnoldi decomposition \cref{eq:AD}, the basic Arnoldi method extracts approximate eigenpairs by the Rayleigh--Ritz process. More specifically, it first computes an ordered Schur decomposition
\begin{equation}
    \label{eq:schurord}
    G_{m}\begin{bmatrix}
        V_{\ell} & V_{\perp}
    \end{bmatrix} = \begin{bmatrix}
        V_{\ell} & V_{\perp}
    \end{bmatrix}\begin{bmatrix}
        S_{\ell} & S_{\ell\perp}\\ 
        0 & S_{\perp}
    \end{bmatrix}.
\end{equation}
The eigenvalues of $G_{m}$ are called Ritz values and, by eigenvalue reordering~\cite{Kressner2006}, it is ensured that the $\ell < m$ \emph{wanted} Ritz values, which approximate eigenvalues of interest, appear on the diagonal of the 
$\ell \times \ell$ upper triangular matrix $S_{\ell}$. The matrix $Q_m V_{\ell}$ is an orthonormal basis for a subspace that approximates the corresponding invariant subspace. Approximate eigenvectors, referred to as Ritz vectors, are obtained by multiplying $Q_m V_{\ell}$ with eigenvectors of $S_{\ell}$.

\paragraph{Restarting with the Krylov--Schur method}
A significant challenge of the basic Arnoldi method arises from the need for storing $Q_m$. Slow convergence leads to large $m$ and the available memory may be exhausted long before satisfactory approximation quality is achieved. In addition, the cost of Gram--Schmidt orthogonalization in the Arnoldi process increases quadratically with $m$. 
Popular algorithms for solving large-scale eigenvalue problems, such as Sorensen's Implicitly Restarted Arnoldi (IRA) method~\cite{sorensen1992implicit} and Stewart's Krylov--Schur method~\cite{stewart2002krylov}, address this challenge by combining the Arnoldi process with restarting. Consider an orthonormal Krylov decomposition
\begin{equation*}
    AQ_{m}=Q_{m}G_{m}+q_{m+1}{c}_{m}^{\Htran},
\end{equation*}
where $[Q_{m},q_{m+1}]\in\C^{n\times (m+1)}$ is orthonormal, $G_{m}$ is not necessarily upper Hessenberg and ${c}_{m}\in\C^{m}$ is a general vector (not necessarily a scalar multiple of $e_{m}$). Given an ordered Schur decomposition~\cref{eq:schurord} of $G_{m}$,
the Krylov--Schur method performs restarting by compressing the orthonormal Krylov decomposition from order $m$ to order $\ell< m$: 
\begin{equation*}
    A(Q_{m}V_{\ell}) = Q_{m}G_{m}V_{\ell}+q_{m+1}c_{m}^{\Htran}V_{\ell} =  (Q_{m}V_{\ell})S_{\ell}+q_{m+1}(V_{\ell}^{\Htran}c_{m})^{\Htran}.
\end{equation*}
After restarting, one can expand an order-$\ell$ Krylov decomposition to order $m$ by means of the Arnoldi process and keep iterating this procedure until convergence.

\paragraph{Randomized Krylov--Schur method} 

To reduce the cost of orthogonalization,
de Damas and Grigori combined the randomized Gram--Schmidt process with IRA~\cite{de2024randomized} and Krylov--Schur~\cite{de2025randomized}. Specifically, replacing the Arnoldi process in the Krylov--Schur method with the randomized Arnoldi process yields randomized Krylov decompositions of the form
\begin{equation*}
    AU_{m} = U_{m}H_{m}+u_{m+1}{c}_{m}^{\Htran},
\end{equation*}
where $[U_{m},u_{m+1}]$ is $\Omega$-orthonormal. Performing an ordered Schur decomposition of $H_{m}$, analogous to \cref{eq:schurord}, we can restart the randomized Krylov decomposition by compressing it from order $m$ to order $\ell < m$. Similarly, after restarting, the decomposition is expanded again to order $m$ using the randomized Arnoldi process, which preserves the $\Omega$-orthonormality of the basis.
Note that the eigenvalues of $H_{m}$ are, in general, \emph{not} Ritz values of $A$.

\subsection{A new randomized Krylov--Schur method that restores similarity} 

Inheriting this drawback from the randomized Arnoldi process discussed in \cref{subsec:newdecomp}, the randomized Krylov--Schur method does not yield a matrix $H_m$ that is similar to the matrix $G_m$ from standard Krylov--Schur. In the following, we propose a new randomized Krylov--Schur method that restores similarity to $G_m$ by leveraging the similarity-restoring randomized Arnoldi process described in \cref{subsec:newrAD}.

\paragraph{Initialization} 
To initialize the first restart cycle, the similarity-restoring randomized Arnoldi method (\cref{algo:RAC}) is applied to obtain an order-$m$ Krylov decomposition of the form \cref{eq:OKD}.
\paragraph{Restarting}
Suppose that we start from an order-$m$ Krylov decomposition 
\begin{equation}
    \label{eq:okdold}
    AU_{m}^{\old} = U_{m}^{\old}\widehat{H}_{m}^{\old} + \widehat{u}_{m+1}^{\old}(c_{m}^{\old})^{\Htran}
    \quad \text{and}\quad  (U_{m}^{\old})^{\Htran}\widehat{u}_{m+1}^{\old}=0.
\end{equation}
Because of \cref{thm:OKD}, the eigenvalues of $\widehat{H}_{m}^{\old}$ are Ritz values of $A$ with respect to $\range\big( U_{m}^{\old} \big)$. Using an ordered Schur decomposition
% of $\widehat{H}_{m}^{\old}$,
\begin{equation}
    \label{eq:Schur}
    \widehat{H}_{m}^{\old}\begin{bmatrix}
        \widehat{V}_{\ell}^{\old} & \widehat{V}_{\perp}^{\old}
    \end{bmatrix}=
    \begin{bmatrix}
        \widehat{V}_{\ell}^{\old} & \widehat{V}_{\perp}^{\old}
    \end{bmatrix}
    \begin{bmatrix}
        \widehat{S}_{\ell}^{\old} & \widehat{S}_{\ell\perp}^{\old}\\ 
        & \widehat{S}_{\perp}^{\old}
    \end{bmatrix},
\end{equation}
the wanted Ritz values are collected in the $\ell \times \ell$ upper triangular matrix $\widehat{S}_{\ell}^{\old}$. Then~\cref{eq:okdold} is compressed into an order-$\ell$ Krylov decomposition by multiplying both sides with $\widehat{V}_{\ell}^{\old}$:
\begin{equation}
    \label{eq:okdnewell}
    \begin{aligned}
        AU_{\ell}^{\new} &= 
    AU_{m}^{\old}\widehat{V}_{\ell}^{\old} =
    U_{m}^{\old}\widehat{H}_{m}^{\old}\widehat{V}_{\ell}^{\old} + \widehat{u}_{m+1}^{\old}(c_{m}^{\old})^{\Htran}\widehat{V}_{\ell}^{\old} \\ 
    &= U_{m}^{\old}\widehat{V}_{\ell}^{\old}\widehat{S}_{\ell}^{\old} + \widehat{u}_{m+1}^{\old}(c_{m}^{\old})^{\Htran}\widehat{V}_{\ell}^{\old}
    =
    U_{\ell}^{\new}\widehat{H}_{\ell}^{\new} + \widehat{u}_{\ell+1}^{\new}(c_{\ell}^{\new})^{\Htran},
    \end{aligned}
\end{equation}
where
\begin{equation*}
    U_{\ell}^{\new}=U_{m}^{\old}\widehat{V}_{\ell}^{\old},\quad \widehat{H}_{\ell}^{\new} = \widehat{S}_{\ell}^{\old},\quad \widehat{u}_{\ell+1}^{\new} = \widehat{u}_{m+1}^{\old},\quad c_{\ell}^{\new} = (\widehat{V}_{\ell}^{\old})^{\Htran}c_{m}^{\old}.
\end{equation*}
\paragraph{Expanding}
To expand \cref{eq:okdnewell} back into an order-$m$ Krylov decomposition, we first perform $m-\ell$ steps of randomized Arnoldi process to obtain 
\begin{equation}
    \label{eq:okdnewm0}
    AU_{m}^{\new} = U_{m}^{\new}H_{m}^{\new} + u_{m+1}^{\new}(c_{m}^{\new})^{\Htran},
\end{equation}
where $U_{m}^{\new}=[U_{\ell}^{\new},\widehat{u}_{\ell+1}^{\new},u_{\ell+2}^{\new},\dotsc,u_{m}^{\new}]$ and $H_{m}^{\new}$ contains
$\widehat{H}_{\ell}^{\new}$ as well as the coefficients from the randomized Gram--Schmidt process. At the end of this cycle, we solve a least-squares problem and use the reformulation~\cref{eq:rKd2oKD} to transform~\eqref{eq:okdnewm0} into the following Krylov decomposition:
\begin{equation}
    \label{eq:okdnewm1}
    AU_{m}^{\new} = U_{m}^{\new}\widehat{H}_{m}^{\new} + \widehat{u}_{m+1}^{\new}(c_{m}^{\new})^{\Htran}
    \quad\text{and}\quad 
    (U_{m}^{\new})^{\Htran}\widehat{u}_{m+1}^{\new} = 0.
\end{equation}
This takes the form~\cref{eq:okdold}, which allows us to start the next cycle.

\paragraph{Sketched reorthogonalization (optional)}

Note that \cref{eq:okdnewm0} is not a randomized Krylov decomposition because $U_{m}^{\new}$ is not $\Omega$-orthonormal; the vector
$\widehat{u}_{\ell+1}^{\new}$ is orthogonal but \emph{not} $\Omega$-orthogonal to $\range(U_{\ell}^{\new})$. This can be fixed by a (sketch) reorthogonalization step.

Assuming that $U_{m}^{\old}$ is $\Omega$-orthonormal, we know that $\widetilde{U}_{m}^{\new}\defi [U_{\ell}^{\new},u_{\ell+2}^{\new},\dotsc,u_{m}^{\new}]$ is $\Omega$-orthonormal. To make $U_{m}^{\new}$ $\Omega$-orthonormal, it suffices to apply one step of randomized Gram--Schmidt to the $(\ell+1)$th column of $U_{m}^{\new}$ and replace it by the vector
\begin{equation}
    \label{eq:optcorrection}
    u_{\ell+1}^{\new} = \frac{1}{\eta}\bigl(\widehat{u}_{\ell+1}^{\new}-\widetilde{U}_{m}^{\new}(\Omega \widetilde{U}_{m}^{\new})^{\dagger}\Omega\widehat{u}_{\ell+1}^{\new}\bigr) 
    =\frac{1}{\eta}\bigl(\widehat{u}_{\ell+1}^{\new}-U_{\ell}^{\new}(\Omega U_{\ell}^{\new})^{\dagger}\Omega\widehat{u}_{\ell+1}^{\new}\bigr),
\end{equation} 
where $\eta$ is a normalization scalar such that $\norm{\Omega u_{\ell+1}^{\new}}=1$.
This additional step only costs $\order(n\ell)$, which is negligible with respect to the other orthogonalization costs.

Let us remark that this step will not change the range of $U_{m}^{\new}$.
By default, we do \emph{not} enable this optional step in our implementation because we observed numerically that $U_{m}^{\new}$ remains well-conditioned even without this step.

\paragraph{Pseudocode}
The procedures described above are summarized in~\cref{algo}. 
\begin{algorithm2e}[htbp]
	\caption{Similarity-restoring randomized Krylov--Schur method} \label{algo}
	\KwIn{Matrix $A\in\C^{n\times n}$. Order $m$ of the Krylov decomposition after expanding and order $\ell<m$ after restarting. Number $k\leq \ell$ of desired eigenvalues. Sketching matrix $\Omega\in\C^{d\times n}$ with $m\leq d\ll n$.}
    \KwOut{Basis $U_{m} \widehat{V}_{k} \in \C^{n\times k}$ of subspace approximating desired invariant subspace.}
    Draw Gaussian random initial vector $b \in\C^{n}$\;
    Compute randomized Krylov decomposition $AU_{m}=U_{m}H_{m}+h_{m+1,m}u_{m+1}e_{m}^{\Htran}$ by performing $m$ iterations of the randomized Arnoldi process\;
    \While{not converged}{
        Solve linear least-squares problem to compute $\widehat{h}_{m}=U_{m}^{\dagger}u_{m+1}$\;
        Update $\widehat{u}_{m+1}=u_{m+1}-U_{m}\widehat{h}_{m}$ and $\widehat{H}_{m}=H_{m}+h_{m+1,m}\widehat{h}_{m}e_{m}^{\Htran}$\;
        Compute partial Schur decomposition $\widehat{H}_{m}\widehat{V}_{\ell}=\widehat{V}_{\ell}\widehat{S}_{\ell}$ corresponding to Ritz values\;
        Check convergence\;
        Set $U_{\ell}=U_{m}\widehat{V}_{\ell}$, $H_{\ell}=\widehat{S}_{\ell}$, $\widehat{u}_{\ell+1}=\widehat{u}_{m+1}$ and $c_{\ell}=(h_{m+1,m}e_{m}^{\Htran}\widehat{V}_{\ell})^{\Htran}$\;
        Expand $AU_{\ell}=U_{\ell}H_{\ell}+\widehat{u}_{\ell+1}c_{\ell}^{\Htran}$ to  $AU_{m}=U_{m}H_{m}+h_{m+1,m}u_{m+1}e_{m}^{\Htran}$ by performing $m-\ell$ steps of the randomized Arnoldi process\;
        (Optional)  Update $(\ell+1)$th column of $U_{m}$ by \cref{eq:optcorrection}\;\label{line:correction} 
    }  
   \Return{$U_{m}\widehat{V}_{k}$}, where $\widehat{V}_{k}$ contains Schur vectors of $\widehat{H}_{m}$ corresponding to $k$ converged eigenvalues\;
\end{algorithm2e}

\subsubsection{Equivalence to standard Krylov--Schur}

In this section, we show that \cref{algo}---in the absence of roundoff errors---produces the same Ritz values as
the standard Krylov--Schur method. Both algorithms thus converge at the same speed to eigenvalues of $A$.
When no restarting is performed, this follows directly from the similarity result established in~\cref{thm:OKD}.

To incorporate restarting, we will consider one cycle of \cref{algo}. The cycle is started with the 
order-$m$ Krylov decomposition~\cref{eq:okdold}. Let us denote the associated orthonormal Krylov decomposition (in the sense of~\cref{thm:OKD}) by
\begin{equation}
    \label{eq:aKDold}
    A\widehat{Q}_{m}^{\old} = \widehat{Q}_{m}^{\old}\widehat{G}_{m}^{\old} + \widehat{u}_{m+1}^{\old}(\widehat{c}^{\old}_{m})^{\Htran}.
\end{equation}
The cycle is completed with the order-$m$ Krylov decomposition~\cref{eq:okdnewm1}. Let us, again, denote the associated orthonormal Krylov decomposition by 
\begin{equation}
    \label{eq:aKDnew}
    A\widehat{Q}_{m}^{\new} = \widehat{Q}_{m}^{\new}\widehat{G}_{m}^{\new} + \widehat{u}_{m+1}^{\new}(\widehat{c}^{\new}_{m})^{\Htran}.
\end{equation}
On the other hand, one can also apply one cycle of the standard Krylov--Schur method \cite{stewart2002krylov}
to the orthonormal Krylov decomposition \cref{eq:aKDold}, which uses the standard Arnoldi process for expansion, to obtain another new
order-$m$ Krylov decomposition.
\Cref{thm:aKD} below shows that in both cases, the new Krylov decomposition is associated with the same (Krylov) subspace. As a consequence, their corresponding Ritz values are the same. This result is illustrated in \cref{fig:aKD}.

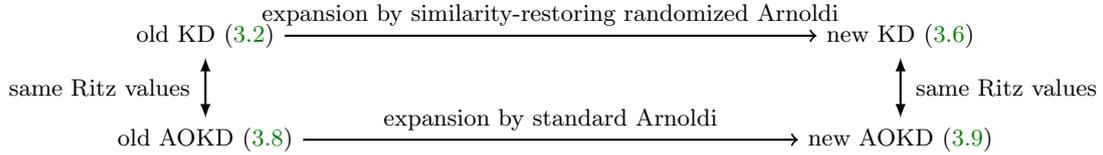
\begin{figure}[htbp]
    \centering
    \begin{tikzpicture}[
        node distance=0.8cm and 7cm,
        every node/.style={font=\small},
        every path/.style={->, thick}
    ]
        % Nodes
        \node (a1) {old KD~\eqref{eq:okdold}};
        \node (a2) [right=of a1] {new KD~\eqref{eq:okdnewm1}};
        \node (b1) [below=of a1] {old AOKD~\eqref{eq:aKDold}};
        \node (b2) [below=of a2] {new AOKD~\eqref{eq:aKDnew}};

        % Arrows
        
        \draw (a1) -- (a2) node[midway, above] {expansion by similarity-restoring randomized Arnoldi};
        \draw (b1) -- (b2) node[midway, above] {expansion by standard Arnoldi};
        \draw[latex-latex] (a1) -- (b1) node[midway, left, xshift=-2pt] {same Ritz values};
        \draw[latex-latex] (a2) -- (b2) node[midway, right, xshift=2pt] {same Ritz values};
    \end{tikzpicture}
    \caption{Equivalence of one cycle of \cref{algo} (using similarity-restoring randomized Arnoldi) 
    and one cycle of Krylov--Schur (using standard Arnoldi) applied to a Krylov decomposition (KD) of the form \cref{eq:OKD} and its associated orthonormal Krylov decomposition (AOKD), respectively. 
    }
    \label{fig:aKD}
\end{figure}

\begin{lemma}
    \label{thm:aKD}
    Consider the Krylov decomposition~\cref{eq:okdold} and 
    assume that the eigenvalues of the diagonal blocks $\widehat{S}_{\ell}^{\old}$ and $\widehat{S}_{\perp}^{\old}$
    in the ordered Schur decomposition~\cref{eq:Schur} are disjoint. Now, consider the 
    orthonormal Krylov decomposition~\cref{eq:aKDold} associated with~\cref{eq:okdold} and a partial Schur decomposition $\widehat{G}_{m}^{\old}\widetilde{V}_{\ell}^{\old} = \widetilde{Y}_{\ell}^{\old}\widetilde{S}_{\ell}^{\old}$
    such that the eigenvalues of $\widehat{S}_{\ell}^{\old}$ and $\widetilde{S}_{\ell}^{\old}$ are the same. Suppose that the corresponding compressed Krylov decomposition,
    \begin{equation}
        \label{eq:KDnew}
        A\widehat{Q}_{m}^{\old}\widetilde{V}_{\ell}^{\old} = \widehat{Q}_{m}^{\old}\widetilde{V}_{\ell}^{\old}\widetilde{S}_{\ell}^{\old}+\widehat{u}_{m+1}^{\old}(\widehat{c}^{\old}_{m})^{\Htran}\widetilde{V}_{\ell}^{\old},
    \end{equation}
    is expanded by performing $m-\ell\geq 1$ steps of the standard Arnoldi process, leading to the order-$m$ Krylov decomposition
    \begin{equation*}
        A\widetilde{Q}_{m}=\widetilde{Q}_{m}\widetilde{G}_{m}+\widetilde{g}_{m+1,m}\widetilde{q}_{m+1}\widetilde c_{m}^{\Htran}.
    \end{equation*}
    Then this Krylov decomposition is equivalent to the orthonormal Krylov decomposition \cref{eq:aKDnew} associated with \cref{eq:okdnewm1}, that is\footnote{The notion  of equivalence established by~\cref{thm:aKD} is slightly stronger than the notion of equivalence defined in \cite[p.~309]{Stewart2001}, which only requires $\range([\widetilde{Q}_{m},\widetilde{q}_{m+1}])=\range([\widehat{Q}_{m}^{\new},\widehat{u}_{m+1}^{\new}])$.} 
    \begin{equation} \label{eq:krylovequivalence}
        \range(\widetilde{Q}_{m}) = \range(\widehat{Q}_{m}^{\new})
        \quad\text{and}\quad 
        \spa{\widetilde{q}_{m+1}} = \spa{\widehat{u}_{m+1}^{\new}}.
    \end{equation}

\end{lemma}
\begin{proof}
By assumptions, $\range(\widetilde{V}_{\ell}^{\old})$ is an invariant subspace of
$\widehat{G}_{m}^{\old}$. From the proof of \Cref{thm:OKD}, we know that
    \begin{equation*}
        \widehat{H}_{m}^{\old} = R^{-1} \widehat{G}_{m}^{\old} R, \quad R= (\widehat{Q}_{m}^{\old})^{\Htran}U_{m}^{\old}, 
    \end{equation*}
    and thus $\range(R^{-1} \widetilde{V}_{\ell}^{\old})$ is an invariant subspace 
    of $\widehat{H}_{m}^{\old}$ belonging to the eigenvalues of
    $\widetilde{S}_{\ell}^{\old}$. But $\range(\widehat{V}_{\ell}^{\old})$
    is also an invariant subspace of $\widehat{H}_{m}^{\old}$ belonging to the same eigenvalues.
    Because $\widehat{S}_{\ell}^{\old}$ and $\widehat{S}_{\perp}^{\old}$ have disjoint eigenvalues, these invariant subspaces are, in fact, identical and, hence, there exists an invertible matrix $T$ such that
    \[
     R^{-1} \widetilde{V}_{\ell}^{\old} = \widehat{V}_{\ell}^{\old} T.
    \]
    This implies 
    \[
     \widehat{Q}_{m}^{\old}\widetilde{V}_{\ell}^{\old} =
     \widehat{Q}_{m}^{\old} R \widehat{V}_{\ell}^{\old} T = 
     \widehat{Q}_{m}^{\old}(\widehat{Q}_{m}^{\old})^{\Htran}U_{m}^{\old}\widehat{V}_{\ell}^{\old}T = U_{m}^{\old}\widehat{V}_{\ell}^{\old}T.
    \]
We therefore have  
    \begin{equation*}
        \begin{aligned}
            \range(\widetilde{Q}_{m})&=\range[\widehat{Q}_{m}^{\old}\widetilde{V}_{\ell}^{\old},\widehat{u}_{m+1}^{\old},A\widehat{u}_{m+1}^{\old},\dotsc,A^{m-\ell-1}\widehat{u}_{m+1}^{\old}] \\
            &= \range[U_{m}^{\old}\widehat{V}_{\ell}^{\old},\widehat{u}_{\ell+1}^{\new},A\widehat{u}_{\ell+1}^{\new},\dotsc,A^{m-\ell-1}\widehat{u}_{\ell+1}^{\new}]
            = \range(\widehat{Q}_{m}^{\new}),
        \end{aligned}
    \end{equation*}
    where we used that expansion adds a Krylov subspace (using the last vector from the previous cycle as initial vector) and $\widehat{u}_{\ell+1}^{\new} = \widehat{u}_{m+1}^{\old}$.
    Moreover,
    \begin{equation*}
        \widetilde{q}_{m+1} = \frac{1}{\widetilde{\eta}} (I-\widetilde{Q}_{m}\widetilde{Q}_{m}^{\Htran})A^{m-\ell}\widehat{u}_{m+1}^{\old} = \frac{1}{\widetilde{\eta}}\bigl(I-\widehat{Q}_{m}^{\new}(\widehat{Q}_{m}^{\new})^{\Htran}\bigr)A^{m-\ell}\widehat{u}_{\ell+1}^{\new} = \alpha \widehat{u}_{m+1}^{\new},
    \end{equation*}
    where $\widetilde{\eta}$ is a normalization scalar and $\abs{\alpha}=1$. This completes the proof.
\end{proof}

For convenience,~\cref{thm:aKD} asssumed that standard Krylov--Schur is applied to the same (associated) orthonormal Krylov  decomposition~\eqref{eq:aKDold}. It is straightforward to verify that the result of the lemma still holds when 
standard Krylov--Schur is applied to an orthonormal Krylov decomposition that is not necessarily identical but only equivalent to~\eqref{eq:aKDold}, in the sense of~\eqref{eq:krylovequivalence}. In particular, this allows us to apply~\cref{thm:aKD} recursively to each cycle of \cref{algo}. It follows that, with the same initial vector, \cref{algo} and the standard Krylov--Schur method generate the same subspace, provided that there is never ambiguity in the choice of the $\ell$  wanted Ritz values.
Recall that the eigenvalues of the compressed matrix $\widehat{H}_{m}$ produced by \cref{algo} are the Ritz values of $A$ with respect to the subspace $\range(U_{m})$. Consequently, both \cref{algo} and the standard Krylov--Schur method always produce the same Ritz values after the same number of cycles.

\subsubsection{Backward stability}

\label{subsec:condition-number-oblique}
In this section, we will assess the impact of roundoff error on the Krylov decompositions produced by~\cref{algo}. Assuming that the optional sketched reorthogonalization step is enabled, let $[\underline{U}_{m},\widehat{\underline{u}}_{m+1}]$ denote the basis \emph{computed} in floating point arithmetic in a cycle of~\cref{algo} after the correction in line~\ref{line:correction} has been performed. Consistent with existing stability results on the randomized Gram--Schmidt process in \cite[Thms.~3.2 and 3.3]{Balabanov2022}, it is not unreasonable to assume that
\begin{equation}
     \label{eq:assume1}
    (\Omega\underline{U}_{m})^{\Htran}(\Omega\underline{U}_{m})-I=\order(\mpr),
\end{equation}
where $\mathbf u$ denotes unit roundoff  ($\approx 10^{-16}$ in double precision).
Here and in the following, $\order(\cdot)$ absorbs mild constants depending only on $m$ and $n$.
We also assume that 
$\Omega$ is an $\varepsilon$-subspace embedding for $\range([\underline{U}_{m},\widehat{\underline{u}}_{m+1}])$ for some $\varepsilon > 0$ not very close to $1$. By~\cref{defn:epsilon-subspace-embedding}, this implies a well-conditioned matrix $\underline{U}_{m}$:
\begin{equation}
     \label{eq:assume2}
     \|\underline{U}_m\| \le \sqrt{1+\varepsilon}, \quad 
     \|\underline{U}_m^\dagger\| \le 1/\sqrt{1-\varepsilon} \quad \Rightarrow \quad 
     \kappa(\underline{U}_{m}) \le \kappa_{\varepsilon} := \sqrt{1+\varepsilon} / \sqrt{1-\varepsilon}.
\end{equation}
In turn, the correction~\eqref{eqn:correction-least-squares-problem} will be computed quite accurately and we may assume that
\begin{equation}
      \label{eq:assume3}
      \norm{\widehat{\underline{u}}_{m+1}}=1+\order(\mpr)
      \quad\text{and}\quad 
      \underline{U}_{m}^{\dagger}\underline{\widehat{u}}_{m+1}=\order(\mpr/\sqrt{1-\varepsilon}).
\end{equation}
Applying the correction and performing matrix-vector products with $A$ will incur further error, which we collect in a residual term $R$ of a perturbed Krylov decomposition
\begin{equation}
      \label{eq:assume4}
 A\underline{U}_{m} = \underline{U}_{m}\underline{\widehat{H}}_{m} +\underline{\widehat{u}}_{m+1}\underline{c}_{m}^{\Htran} + R, \quad 
 R = \order(\kappa_{\varepsilon}\sqrt{1+\varepsilon}\|A\| \mpr), \quad \underline{c}_{m} = \order(\sqrt{1+\varepsilon}\|A\|),
\end{equation}
similar to the roundoff error model considered for the standard Krylov--Schur method in~\cite[Eq~(4.1)]{stewart2002krylov}. Under these assumptions, the following theorem establishes a backward stability result. 
\begin{theorem}
Suppose that the assumptions~\cref{eq:assume1,eq:assume2,eq:assume3,eq:assume4} are satisfied with $\kappa_{\varepsilon} \ll 1/\mpr$.
Then there exist a vector $\widehat{u}_{m+1}$ and a matrix $E$ with
\[
 \widehat{u}_{m+1}-\underline{\widehat{u}}_{m+1} =\order(\kappa_{\varepsilon}\mpr), \quad E = 
 \order(\norm{A}\kappa_{\varepsilon}^{2}\mpr),
\]
such that $\norm{\widehat{u}_{m+1}}=1$ and 
\begin{equation}
        \label{eq:thmbs2new}
        (A+E)\underline{U}_{m}=\underline{U}_{m}\underline{\widehat{H}}_{m}+\widehat{u}_{m+1}\underline{c}_{m}^{\Htran},
        \quad \underline{U}_{m}^{\Htran}\widehat{u}_{m+1}=0.
\end{equation}
\end{theorem}
\begin{proof}
    Setting
    \[
     \widehat{u}_{m+1} = \frac{(I-\underline{U}_{m}\underline{U}_{m}^{\dagger})\underline{\widehat{u}}_{m+1}}{\norm{(I-\underline{U}_{m}\underline{U}_{m}^{\dagger})\underline{\widehat{u}}_{m+1}}}
    \]
    implies $\underline{U}_{m}^{\Htran}\widehat{u}_{m+1}=0$. Moreover, assumption~\eqref{eq:assume3}
    yields $\widehat{u}_{m+1} - \underline{\widehat{u}}_{m+1} = \order(\kappa_{\varepsilon}\mpr)$. Setting
    \[
     E = \big( -R - \underline{\widehat{u}}_{m+1}\underline{c}_{m}^{\Htran} + \widehat{u}_{m+1}\underline{c}_{m}^{\Htran} \big) \underline{U}_{m}^\dagger,
    \]
    one directly verifies that~\eqref{eq:thmbs2new} holds. Moreover, 
    \[
     \|E\| \le \|\underline{U}_{m}^\dagger\| \big( \|R\| + \|\widehat{u}_{m+1} - \underline{\widehat{u}}_{m+1}\| \|\underline{c}_m\|\big) = \order( \norm{A}\kappa_{\varepsilon}^{2}\mpr),
    \]
where we used assumptions~\cref{eq:assume2,eq:assume4}.
\end{proof}

Combined with \Cref{thm:OKD}, the backward error result~\eqref{eq:thmbs2new} implies, among others, that the eigenvalues of $\underline{\widehat{H}}_{m}$ are Ritz values of a slightly perturbed matrix $A$.

\section{Application to matrix functions}
\label{sec:matrix-functions}

Given $A \in \mathbb C^{n\times n}$ and $b\in \C^n$, 
this section aims at approximating $f(A)b$ by means of the similarity-restoring randomized Arnoldi process from \cref{subsec:newrAD}. We assume that $f: \region \to \mathbb C$ is defined on a domain $\region$ that contains the numerical range $W(A)$ of $A$ and refer to~\cite{Higham08} for the definition of $f(A)$.

\subsection{Arnoldi and randomized Arnoldi for matrix functions}

We start by recalling existing (randomized) Arnoldi methods for $f(A) b$. 
Given an Arnoldi decomposition~\cref{eq:AD} with an orthonormal basis $Q_m$, the Arnoldi method for matrix functions~\cite{DruskinKnizhnerman89,Saad92} returns the approximation:
\begin{equation}
	\label{eqn:fAq-arnoldi-orth}
	f_m \defi Q_m f(G_m) \norm{b} e_1, \quad \text{where} \quad Q_m^\Htran Q_m = I \quad \text{and} \quad G_m = Q_m^\Htran A Q_m.
\end{equation}
Note that $f(G_m)$ is well defined because $W(G_{m}) \subset W(A)$.
In \cref{eqn:fAq-arnoldi-orth} and in the following, we assume that $f(G_m)$ and the other matrix functions that appear are well-defined.
Given any basis $W_m$ of $\kryl_m(A, b)$, not necessarily orthonormal, we can alternatively write $f_m$ as 
\begin{equation}
	\label{eqn:fAq-arnoldi-general}
	f_m = W_m f(W_m^\dagger A W_m) W_m^\dagger b.
\end{equation}
If the first column of $W_m$ is $\alpha b$ for some scalar $\alpha \in \C$, this expression simplifies to $f_m = W_m f(W_m^\dagger A W_m) \alpha^{-1} e_1$. 

In \cite{GuettelSchweitzer23,CKN24,PSS25}, the randomized Krylov decomposition \cref{eq:rAD} with $u_1 = b / \norm{\Omega b}$ is used to return the following approximation of $f(A) b$:
\begin{equation}
	\label{eqn:fAq-arnoldi-rand}
	f_m^\Omega \defi U_m f(H_m) \norm{\Omega b} e_1, \quad \text{where} \quad (\Omega U_m)^\Htran \Omega U_m = I \quad \text{and} \quad H_m = (\Omega U_m)^\Htran \Omega A U_m.
\end{equation}
This approximation is called \emph{sketched FOM approximation} in \cite{PSS25}. 
In this case, given any basis $W_m$ of $\kryl_m(A, b)$, not necessarily $\Omega$-orthonormal, we can rewrite~\cref{eqn:fAq-arnoldi-rand} as
\begin{equation*}
	f_m^\Omega \defi W_m f((\Omega W_m)^\dagger \Omega A W_m) (\Omega W_m)^\dagger \Omega b,
\end{equation*} 
see, for instance, \cite[Eq.~(9)]{PSS25}. 
An explicit expression for the difference between $f_m$ and $f_m^\Omega$ has been given in \cite[Cor.~5.1]{PSS25}, and it has been used to derive error bounds such as, for example, \cite[Cor.~5.4]{PSS25}. Other theoretical results on the convergence of the approximation $f_m^\Omega$ can be found in \cite[Sec.~4]{CKN24} and \cite[Sec.~2.3]{GuettelSchweitzer23}. 
Despite these results in the recent literature, the current theoretical understanding of the convergence of the approximation \cref{eqn:fAq-arnoldi-rand} is not completely satisfactory. One serious obstacle is that the numerical range of the projected matrix $H_m$ in~\cref{eqn:fAq-arnoldi-rand} may be significantly larger than the numerical range of $G_m$ in \cref{eqn:fAq-arnoldi-orth} and may, in fact, not be contained in $W(A)$. In an extreme scenario, this might mean that $f(H_m)$ is not well-defined for a function $f$ with singularities outside $W(A)$. More importantly, it leads to spikes in the convergence of the randomized Arnoldi approximation \cref{eqn:fAq-arnoldi-rand}. This kind of behavior has been observed, for example, in \cite{GuettelSchweitzer23}; see also \cref{subsubsec:experiments-matfun-convergence}.

\subsection{Restoring similarity in randomized Arnoldi for matrix functions}
\label{subsec:corrected-arnoldi-matrix-functions}

Let us now consider a Krylov decomposition of the form~\cref{eq:OKD}, additionally assuming that $U_m$ is $\Omega$-orthonormal and $c_m = h_{m+1,m}e_m$:
\begin{equation*}
	A U_m = U_m \widehat{H}_m + h_{m+1, m}\widehat{u}_{m+1} e_m^\Htran, \quad\text{and}\quad (\Omega U_m)^\Htran \Omega U_m = I, \quad U_m^\Htran \widehat{u}_{m+1} = 0.
\end{equation*}
Then
\begin{equation*}
	U_m^\dagger A U_m = \widehat{H}_m + h_{m+1,m} U_m^\dagger \widehat{u}_{m+1} e_m^\Htran = \widehat{H}_m.
\end{equation*}
Using $u_1 = b / \norm{\Omega b}$ and~\cref{eqn:fAq-arnoldi-general} with $W_m = U_m$ gives us the following equivalent expression:
\begin{equation}
	\label{eqn:fAq-arnoldi-corrected}
	f_m = U_m f(\widehat{H}_m) \norm{\Omega b} e_1.
\end{equation}
Let us stress that this approximation is identical to the standard Arnoldi approximation~\eqref{eqn:fAq-arnoldi-orth} and, thus, classical convergence results from the literature, such as~\cite{Beckermann2009,Hochbruck1997}, apply.

The expression~\cref{eqn:fAq-arnoldi-corrected} allows us to efficiently compute the Arnoldi approximation \cref{eqn:fAq-arnoldi-orth} even if the basis $U_m$ is not orthonormal, we thus benefit from the more robust convergence behavior of~\cref{eqn:fAq-arnoldi-orth} at a minor increase in computational cost relative to~\cref{eqn:fAq-arnoldi-rand}.
If the accuracy of $f_m$ is not satisfactory and we need to perform additional Arnoldi iterations, we simply discard the corrected quantities $\widehat{H}_m$ and $\widehat{u}_{m+1}$ and construct an expanded randomized Krylov decomposition~\cref{eq:rAD}, ensuring that the basis is still $\Omega$-orthonormal until the next time that we apply the similarity-restoring correction in \cref{eq:OKD} to compute an approximation to $f(A)b$. 
Since the cost of solving the linear least-squares problem \cref{eqn:correction-least-squares-problem} can quickly add up, we recommend to compute the approximation to $f(A)b$ using~\cref{eqn:fAq-arnoldi-corrected} only once every few iterations. The resulting algorithm is summarized in \cref{algo:matfun}, where the approximation~\cref{eqn:fAq-arnoldi-corrected} to $f(A) b$ is computed once every $\ell$ iterations.

\begin{remark} 
	\label{rem:matfun-comparison-with-ckn}
	The expression \cref{eqn:fAq-arnoldi-corrected} for the approximation of $f(A) b$ has already been considered in \cite{CKN24}. In particular, if the least-squares problem at line 4 of \cref{algo:matfun} is solved with LSQR with tolerance $10^{-6}$, then \cref{algo:matfun} becomes identical to \cite[Algo.~3.1]{CKN24}. 
	When this least-squares problem is solved inexactly, the difference between $f_m$ and the computed approximation can be bounded with \cite[Lem.~4.2]{CKN24}.
\end{remark}

\begin{algorithm2e}[t]
	\caption{Similarity-restoring randomized Arnoldi for matrix functions} \label{algo:matfun}
	\KwIn{Matrix $A\in\C^{n\times n}$, vector $b \in \C^n$, function $f$, sketching matrix $\Omega\in\C^{d\times n}$ with $M\leq d\ll n$, maximum number of Arnoldi iterations $M$, integer $\ell \ge 1$.}
    \KwOut{Approximation $f_m \approx f(A) b$ from \cref{eqn:fAq-arnoldi-corrected} for some $m \le M$.}
	Set $m = \ell$\; 
    Compute randomized Krylov decomposition $AU_{m}=U_{m}H_{m}+h_{m+1,m}u_{m+1}e_{m}^{\Htran}$ by performing $m$ iterations of the randomized Arnoldi process with initial vector $b$\;
    \While{$m \le M$}{
        Solve linear least-squares problem to compute $\widehat{h}_{m}=U_{m}^{\dagger}u_{m+1}$\;
        Update $\widehat{u}_{m+1}=u_{m+1}-U_{m}\widehat{h}_{m}$ and $\widehat{H}_{m}=H_{m}+h_{m+1,m}\widehat{h}_{m}e_{m}^{\Htran}$\;
		Compute the approximation $f_m = U_m f(\widehat{H}_m) \norm{\Omega b} e_1$\;
		\If{converged}{
			\Return $f_m$\;
		}
		Perform $\ell$ iterations of randomized Arnoldi to expand randomized Krylov decomposition to $AU_{m+\ell}=U_{m+\ell}H_{m+\ell}+h_{m+\ell+1,m+\ell}u_{m+\ell+1}e_{m+\ell}^{\Htran}$\;
		Set $m = m + \ell$\;  
    }  
\end{algorithm2e}

\section{Numerical experiments}
\label{sec:numerical-experiments}

In this section, we compare our similarity-restoring algorithms numerically against standard and randomized algorithms from the literature for the computation of eigenvalues and matrix functions. 
All numerical experiments in this section have been implemented in Matlab 2022b and were carried out on an AMD Ryzen~9 6900HX Processor (8 cores, 3.3--4.9 GHz) and 32 GB of RAM. Scripts to reproduce numerical results are publicly available at \url{https://github.com/nShao678/Similarity-restoring-randomized-Arnoldi-code}.

The randomized Gram--Schmidt process can be implemented in several different ways; see, for example, \cite{Balabanov2022,BalabanovGrigori25block}. In its $k$th iteration, let us denote by $U_k$ an $\Omega$-orthonormal basis and by $w$ the new vector that needs to be $\Omega$-orthonormalized.
In our experiments, we use the following two variants for this calculation.
\begin{description}
\label{des:RGS}
 \item[RGS:] Use a QR decomposition of $\Omega U_k$ to compute $\widetilde{u}_{k+1} = w - U_k (\Omega U_k)^\dagger \Omega w$. This approach corresponds, for instance, to the \emph{backslash} operator in MATLAB.
 \item[RCGS2:] Apply a randomized classical Gram--Schmidt process twice: first compute $\widetilde{w} = w - U_k (\Omega U_k)^\Htran \Omega w$ and then compute $\widetilde{u}_{k+1}$ by performing an $\Omega$-reorthogonalization, as $\widetilde{u}_{k+1} = \widetilde{w} - U_k (\Omega U_k)^\Htran \Omega \widetilde{w}$.
\end{description}
In both implementations, the new basis vector $u_{k+1}$ is obtained with the $\Omega$-normalization $u_{k+1} = \widetilde{u}_{k+1} / \norm{\Omega \widetilde{u}_{k+1}}$.

\subsection{Eigenvalue problems}
\label{subsec:numexp-eigenvalue-problems}

In this section, we compare \cref{algo} (SRR-KS) with the standard Krylov--Schur method (KS) and the randomized Krylov--Schur method (RKS).

Given a matrix $A\in\C^{n\times n}$, we compute $k$ extremal eigenvalues. We perform a restarting when the dimension of the Krylov subspace reaches $m = 4k$, and we compress the Krylov decomposition to order $\ell = 2k$, except in \cref{sec:expdMax}. A sparse sign sketching matrix $\Omega \in \R^{d\times n}$ is applied to the randomized Krylov--Schur method and \cref{algo}. The initial vector is always a Gaussian random vector and the stopping criterion is when the residual norms are below $10^{-7}$. No deflation is performed.
In the standard Krylov--Schur method, the basis vectors are orthogonalized by performing classical Gram--Schmidt with reorthogonalization (CGS2), while in randomized Krylov--Schur and \cref{algo} the vectors are $\Omega$-orthogonalized with the RGS implementation of the randomized Gram--Schmidt method \cref{eq:randgs}.

For \cref{algo}, we do not perform the optional sketched reorthogonalization step described in \cref{eq:optcorrection} in any of the experiments.
Even without this step, we do not observe any stability issues, and the bases $U_{m}$ remain consistently well conditioned in all tests.

\subsubsection{Comparison on execution time}
\label{subsubsec:numexp-eig-time}

Inspired by the test matrices used in \cite[Sec.~5.1]{de2025randomized}, we consider the following matrix:
\begin{equation}
    \label{eq:defAn}
    A_{n} = F_{n}^{-1}\begin{bmatrix}
        f(a_{1}) & g_{1} &&&&\\ 
        &f(a_{2}) & g_{2} &&&\\ 
        &&\ddots &\ddots &&\\ 
        &&&f(a_{n-1}) &g_{n-1}\\ 
        &&&&f(a_{n})
    \end{bmatrix}F_{n},
\end{equation}
where $F_{n}$ is the discrete Fourier transformation, $g_{i}$ are i.i.d.~standard Gaussian random variables for non-Hermitian examples and $0$ for Hermitian examples, and $a_{i}$ are equispaced points from $2$ to $10$. 
We compute the $k=10$ eigenvalues of $A_n$ with largest magnitude.
Consider the following four functions $f$:
\begin{equation}
    \label{eq:deffi}
    f_{1}(a) = \exp(a/10),\quad f_{2}(a) = \log(a+1),\quad f_{3}(a) = 1+\frac{1}{a^{2}},\quad f_{4}(a) = 0.99^{a}.
\end{equation}
We set $d=100$ and range $n$ in $\{10000,20000,40000,80000,160000\}$.
The execution times are collected in \cref{fig:time}, showing that both RKS and SRR-KS only take about half the time of KS, and the difference in execution time of RKS and SRR-KS is very small.
The least-squares problem \cref{eqn:correction-least-squares-problem} is solved with either Cholesky or LSQR with tolerance $10^{-12}$. In this test, LSQR is always slightly more expensive than Cholesky, since the matrix $U_m$ has a relatively small number of columns.

\begin{figure}[t]
    \centering
    \subfloat[Non-Hermitian matrices]{
    \includegraphics[width = \figsizeQ]{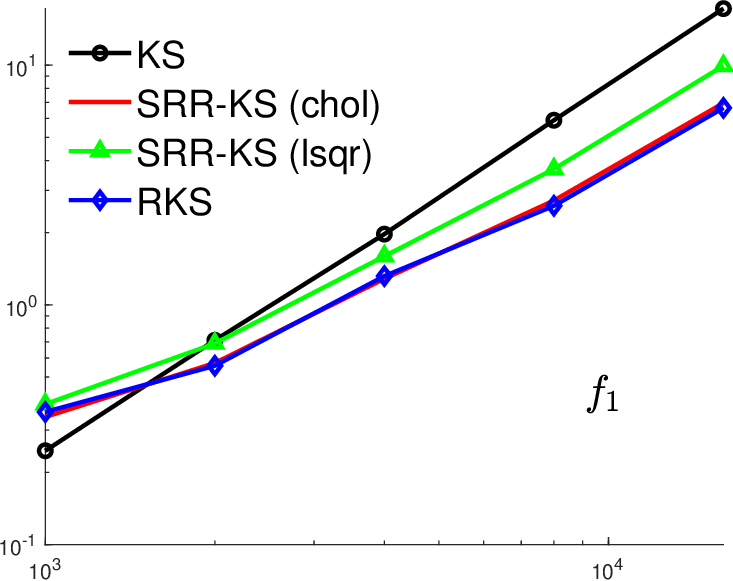}
    \includegraphics[width = \figsizeQ]{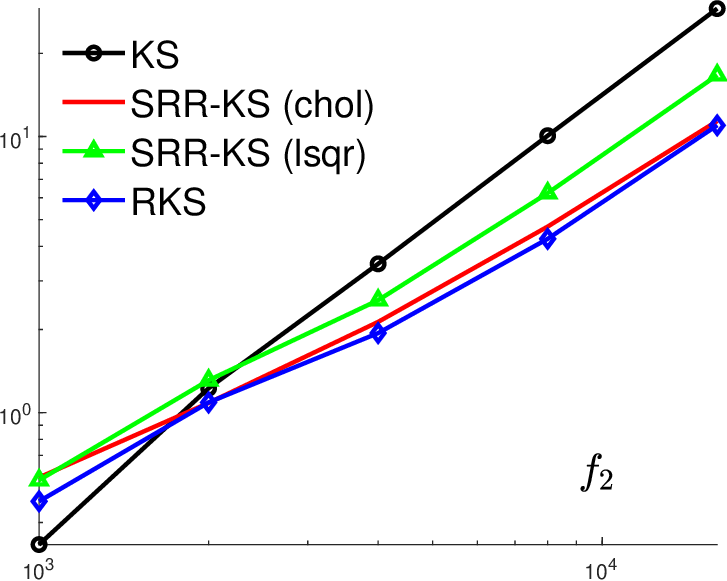}
    \includegraphics[width = \figsizeQ]{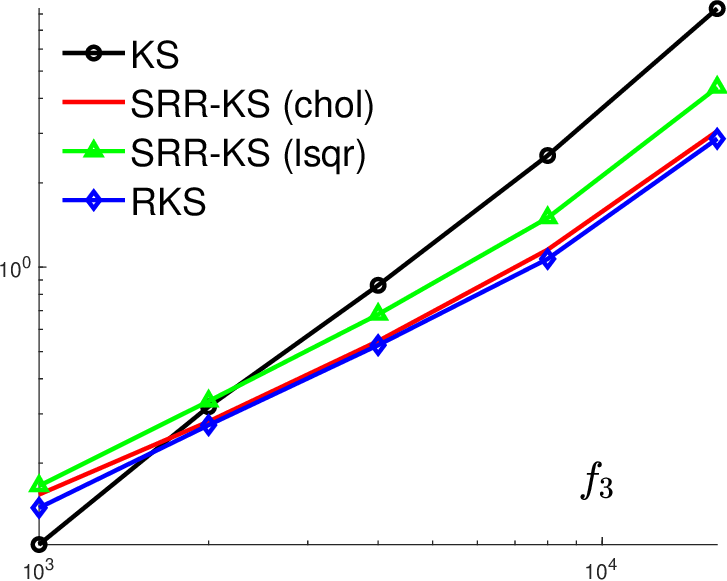}
    \includegraphics[width = \figsizeQ]{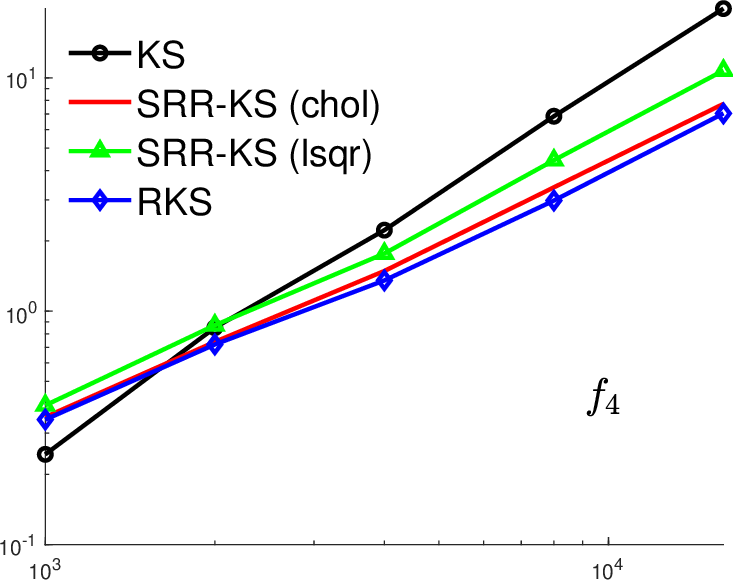}
    }

    \subfloat[Hermitian matrices]{
    \includegraphics[width = \figsizeQ]{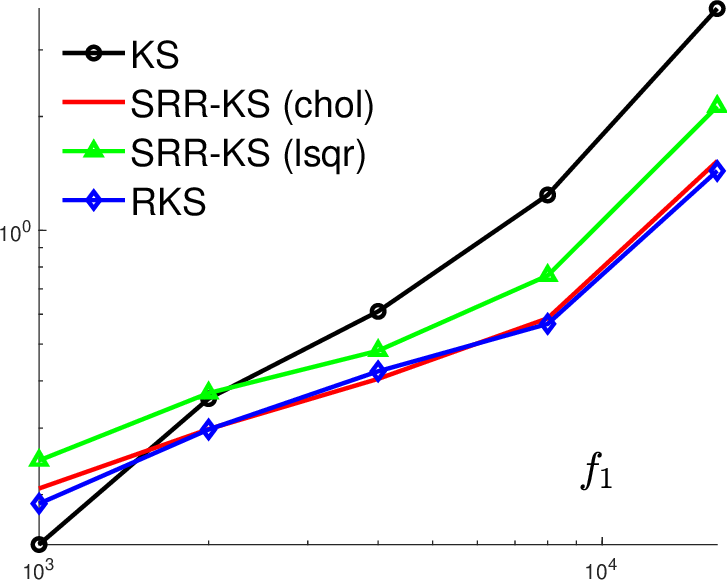}
    \includegraphics[width = \figsizeQ]{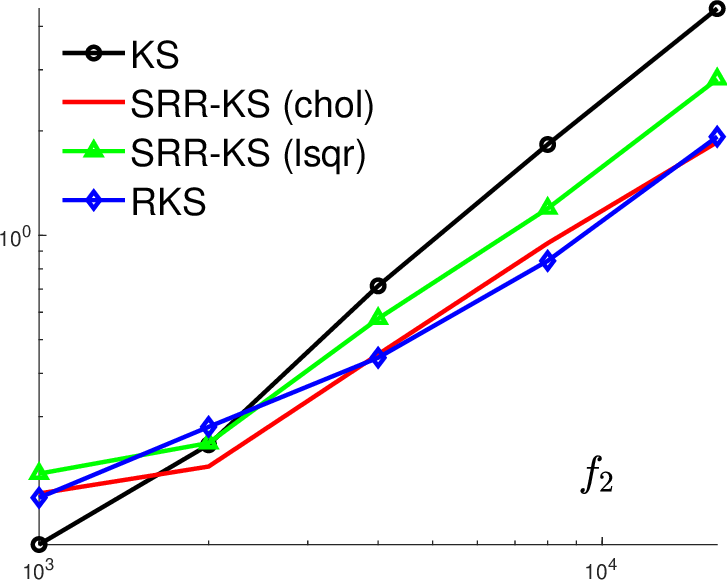}
    \includegraphics[width = \figsizeQ]{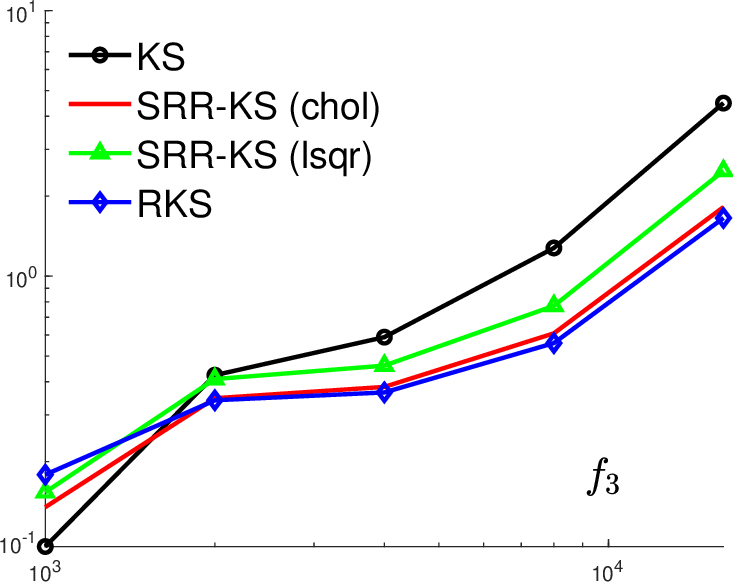}
    \includegraphics[width = \figsizeQ]{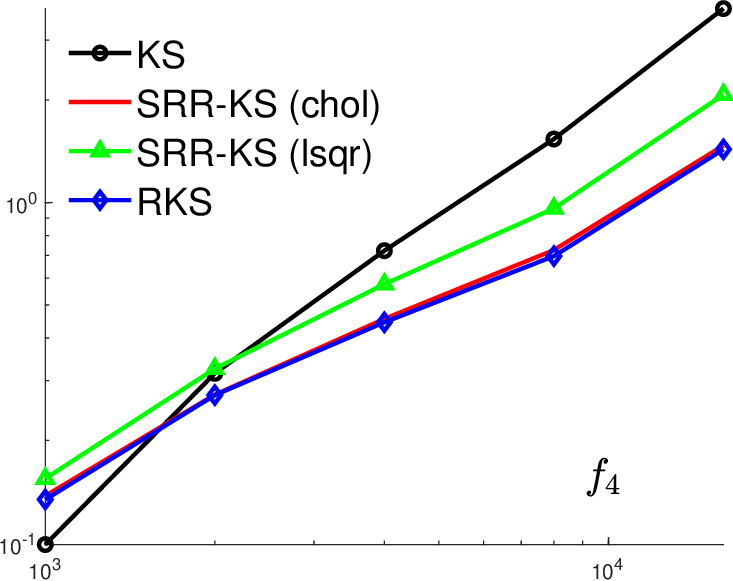}}
    \caption{Execution time (s) of Krylov--Schur (KS), \cref{algo} (SRR-KS) with Cholesky and LSQR inner least-squares solvers, and randomized Krylov--Schur (RKS) versus matrix sizes.
    From left to right, the eigenvalues are of type $f_{1}$ to $f_{4}$ in \cref{eq:deffi} for the non-Hermitian (top) and Hermitian (bottom) example matrix~\eqref{eq:defAn}. 
	\label{fig:time}}
    
\end{figure}

\subsubsection{Comparison with different $m$ and $\ell$}
\label{sec:expdMax}
The computational cost of Krylov--Schur methods mainly comes from two parts: applying the matrix $A$ to a vector and (re)orthogonalization. In general, using a higher order Krylov decomposition will reduce the number of matvecs with $A$ but increase the cost for orthogonalization. This trade-off makes the maximum dimension of Krylov subspace an important, but also sensitive parameter in the Krylov--Schur method.
As shown in \cite{de2025randomized,de2024randomized}, this parameter becomes less sensitive in RKS because the cost for orthogonalization is significantly reduced by the randomized Gram--Schmidt process. The following experiment shows that \cref{algo} also has this property.

Consider the matrix $A_{n}$ in \cref{eq:defAn}, where $n=40010$.
For $10000(k-1) <i \leq 10000k $ with $1\leq k\leq 4$, we pick $f(a_{i})$ as Gaussian random variables with distribution $\mathcal{N}(10^{k},10^{k-1})$, where the last $10$ eigenvalues are from $\mathcal{N}(0,1)$. Those eigenvalues forms $5$ clusters, and we compute the last cluster containing $k=10$ eigenvalues with smallest real parts. Denote the order of the Krylov decomposition before and after restarting by $m$ and $\ell$ respectively. We range $\ell$ in $\{10,15,20,25\}$ and $m$ in $\{30,40,50,60\}$, and use dimension $d=100$ for the sparse sign sketching. For \cref{algo}, in light of the observations in \cref{subsubsec:numexp-eig-time} we always solve the least-squares problem \cref{eqn:correction-least-squares-problem} with a Cholesky decomposition.

\begin{figure}[t]
    \centering
        \includegraphics{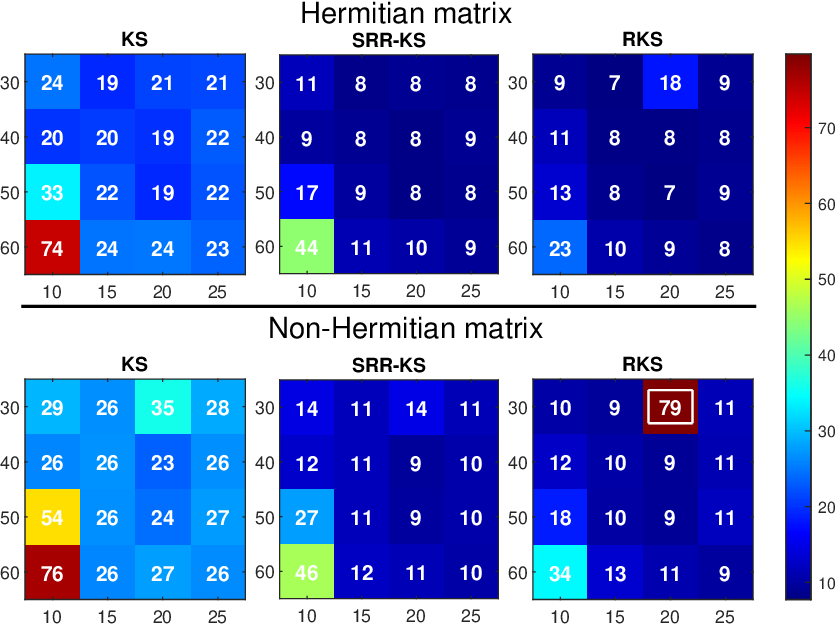}
        \caption{Execution time (s) for KS (left), SRR-KS (middle) and RKS (right) on Hermitian (top) and non-Hermitian (bottom) problems. The X-label and Y-label are $\ell$ and $m$, the dimension of Krylov subspaces before and after restarting, respectively.
        The stopping criterion is when the residual norms are below $10^{-7}$.
        The RKS method applied to the non-Hermitian problem with $\ell=20$ and $m=30$ does not converge after 10000 matvecs. 
	\label{fig:dMax}}
\end{figure}

The execution times for all three methods are collected in \cref{fig:dMax}, showing that KS takes significantly more time than SRR-KS and RKS. Moreover, when computing $k=10$ eigenvalues, selecting $\ell=k=10$ is not effective, especially when we can pick a slightly larger $m$. The recommended parameters are $m=2\ell=4k$. For RKS and SRR-KS, since their execution times for one cycle are close (SRR-KS is slightly slower than RKS due to the least-squares correction), the execution time also reflects the number of cycles. Except for some extreme choices such as $\ell=10$ and $m=60$, both SRR-KS and RKS are slightly more robust than KS with respect to the choice of the parameters $m$ and $\ell$.

When $\ell=k=10$, RKS sometimes needs fewer restart cycles than SRR-KS. But for larger $\ell$, such as the recommended choice $\ell=2k=20$, RKS may require (significantly) more cycles for convergence than KS and SRR-KS for both Hermitian and non-Hermitian problems. To further illustrate this phenomenon, we plot the convergence histories for $\ell = 20$ and $m = 30$ in \cref{fig:hist}. For both Hermitian and non-Hermitian cases, the residual of SRR-KS behaves similarly to that of KS, but the behavior of RKS is bad. Although RKS initially follows a pattern similar to KS, it soon begins to deteriorate, leading to poor convergence in both cases.

\begin{figure}[t]
    \centering
        \includegraphics[width=\figsizeD]{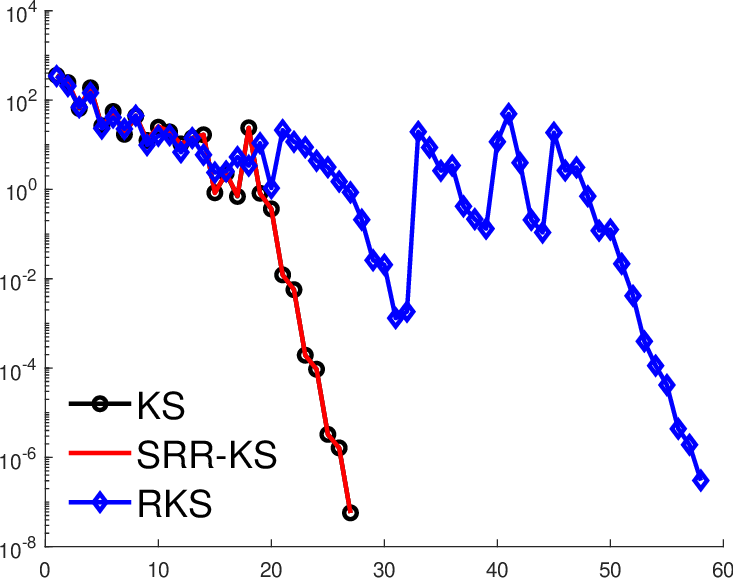}
        \includegraphics[width=\figsizeD]{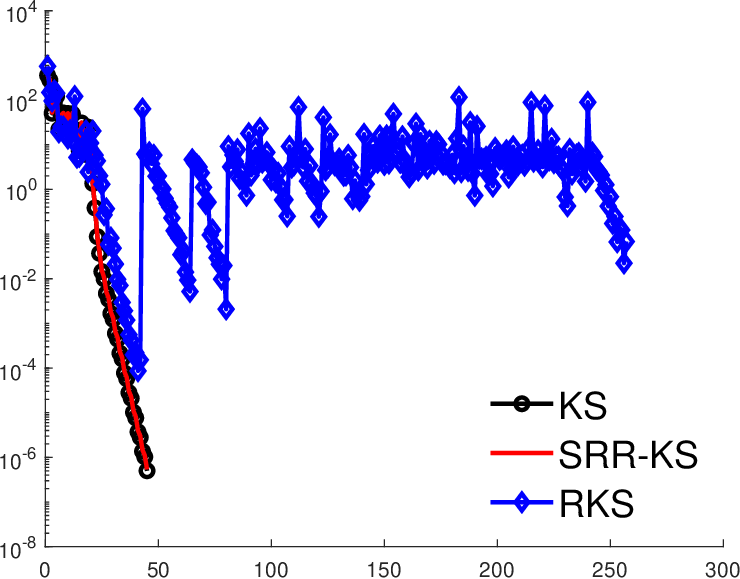}
        \caption{Convergence history of residual for KS,  SRR-KS and RKS on Hermitian (left) and non-Hermitian (right) problems with $\ell=20$ and $m=30$. The history of KS and SRR-KS are overlapped.
	\label{fig:hist}}
\end{figure}

\subsection{Matrix functions}
\label{subsec:numerical-experiments-matfun}

In this section, we compare \cref{algo:matfun} (denoted by SRR-Arnoldi) with the Arnoldi approximation \cref{eqn:fAq-arnoldi-orth} and the randomized Arnoldi approximation \cref{eqn:fAq-arnoldi-rand} 
for $f(A)b$ in terms of convergence rate and execution time.
The orthogonalization process in the Arnoldi approximation \cref{eqn:fAq-arnoldi-orth} is based on CGS2. 
For the randomized Gram--Schmidt process in the randomized Arnoldi approximation \cref{eqn:correction-least-squares-problem} and SRR-Arnoldi, we employ the RCGS2 implementation described at the beginning of \cref{sec:numerical-experiments}.

\subsubsection{Convergence for symmetric matrices}
\label{subsubsec:experiments-matfun-convergence}

We consider a symmetric test matrix $A$ inspired by the matrix \texttt{G-clust5-s25} used in \cite{TGB23}. Letting $n = 10000$, we define $A = Q^{\Ttran} D Q \in \R^{n \times n}$, where $Q$ is the (orthogonal) discrete cosine transform matrix, and $D$ is a diagonal matrix whose diagonal entries $d_i$ are defined as follows: for $1 \le k \le 4$, the diagonal entries $d_i$ with $1 + 2500(k-1)\le i \le 2500k$ are given by independent Gaussian random variables with distribution $\mathcal{N}(10^{k-1}, 10^{k-2})$. In other words, the eigenvalues of $A$ are grouped into four clusters of equal size, centered at $\{1,10,100,1000\}$. This kind of eigenvalue distribution gives rise to frequent spikes in the convergence history of randomized FOM \cite[Fig.~4]{TGB23} for linear systems, so we expect to observe a similar behavior also for the approximation of $f(A) b$. For this test, we do not exploit the symmetry of $A$ in the standard Arnoldi approximation, and we use full orthogonalization to generate the basis $Q_m$.

We approximate $f(A) b$, where $b$ is a Gaussian random vector, and $f(z)$ is either $z^{1/2}$, $z^{-1/2}$ or $\log z$.  
We use $m = 300$ Arnoldi iterations for all methods and a sketching size $d = 2m$, and compare three different approaches to solving the least-squares problem \cref{eqn:correction-least-squares-problem} in SRR-Arnoldi: Cholesky and LSQR with tolerances $10^{-12}$ and $10^{-1}$. The convergence history for the different approximations is shown in \cref{fig:fAb-comparison}, together with the ratios of the errors of each method with respect to the standard Arnoldi approximation, for a better visualization of the spikes. When \cref{eqn:correction-least-squares-problem} is solved with Cholesky or LSQR with tolerance $10^{-12}$, the error curve for SRR-Arnoldi  completely overlaps with the standard Arnoldi approximation, as expected. 
In this test, the randomized Arnoldi approximation \cref{eqn:fAq-arnoldi-rand} exhibits several spikes in its error, which decrease its robustness. Even if we use a very loose LSQR tolerance of $10^{-1}$ for solving \cref{eqn:correction-least-squares-problem} in SRR-Arnoldi, we observe that the magnitude of the spikes is significantly reduced with respect to randomized Arnoldi. 

Although SRR-Arnoldi was already compared against other deterministic and randomized methods in \cite[Sec.~5]{CKN24}, the experiments in this section allow us to attain a more comprehensive understanding of the convergence behavior and highlight the robustness issues of the randomized Arnoldi approximation \cref{eqn:fAq-arnoldi-rand}, explicitly showing that SRR-Arnoldi is able to correct them. In addition, while \cite{CKN24} always employs LSQR with tolerance $10^{-6}$ to solve \cref{eqn:correction-least-squares-problem}, by comparing the behavior of different methods for the solution of the least-squares problem we observe that even a very cheap solver with a loose tolerance (such as LSQR with tolerance $10^{-1}$) is able to strongly mitigate the magnitude of the spikes in the convergence of the randomized Arnoldi approximation.

\begin{figure}[t]
    \centering
    \includegraphics[width = \figsizeT]{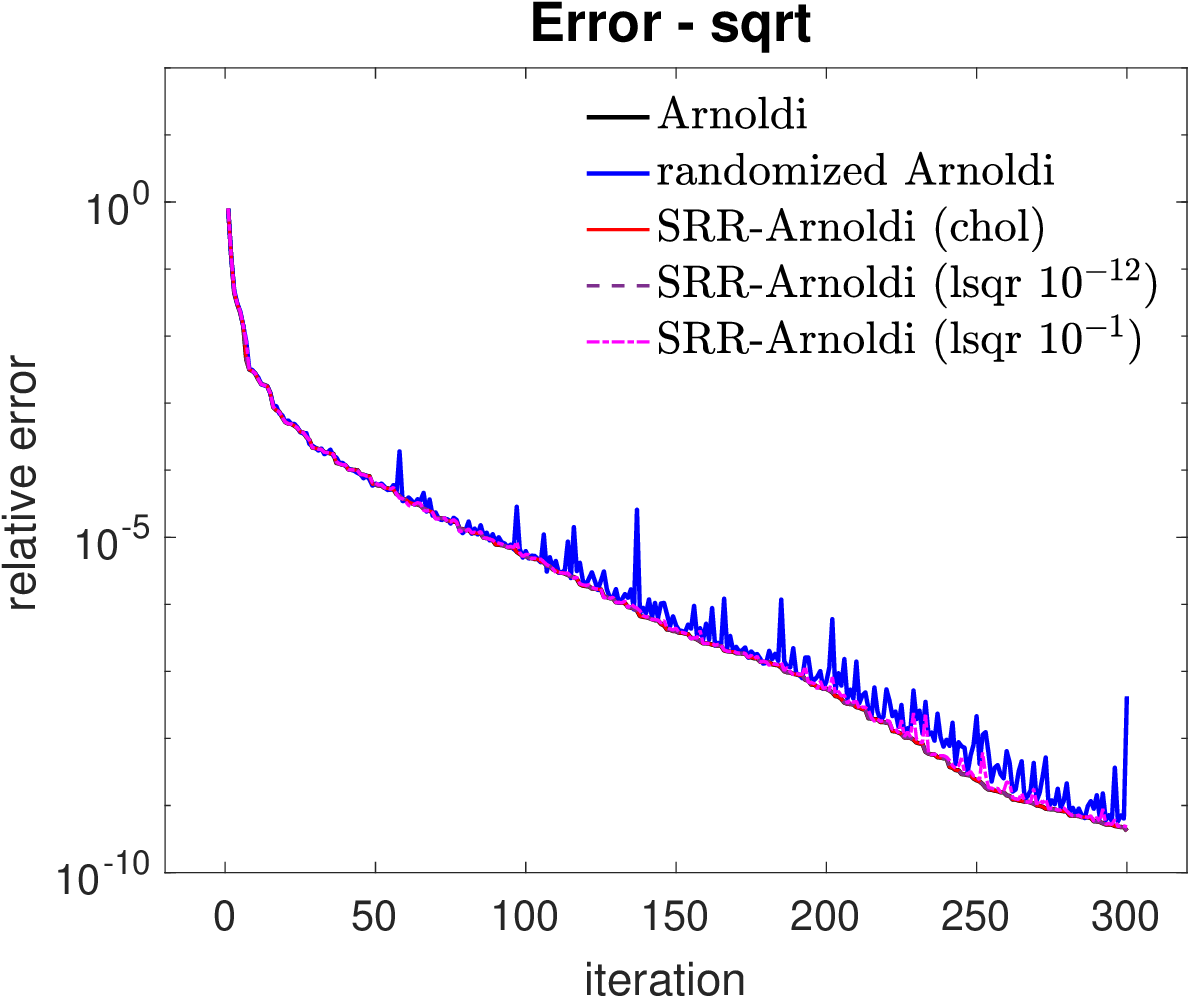}
    \includegraphics[width = \figsizeT]{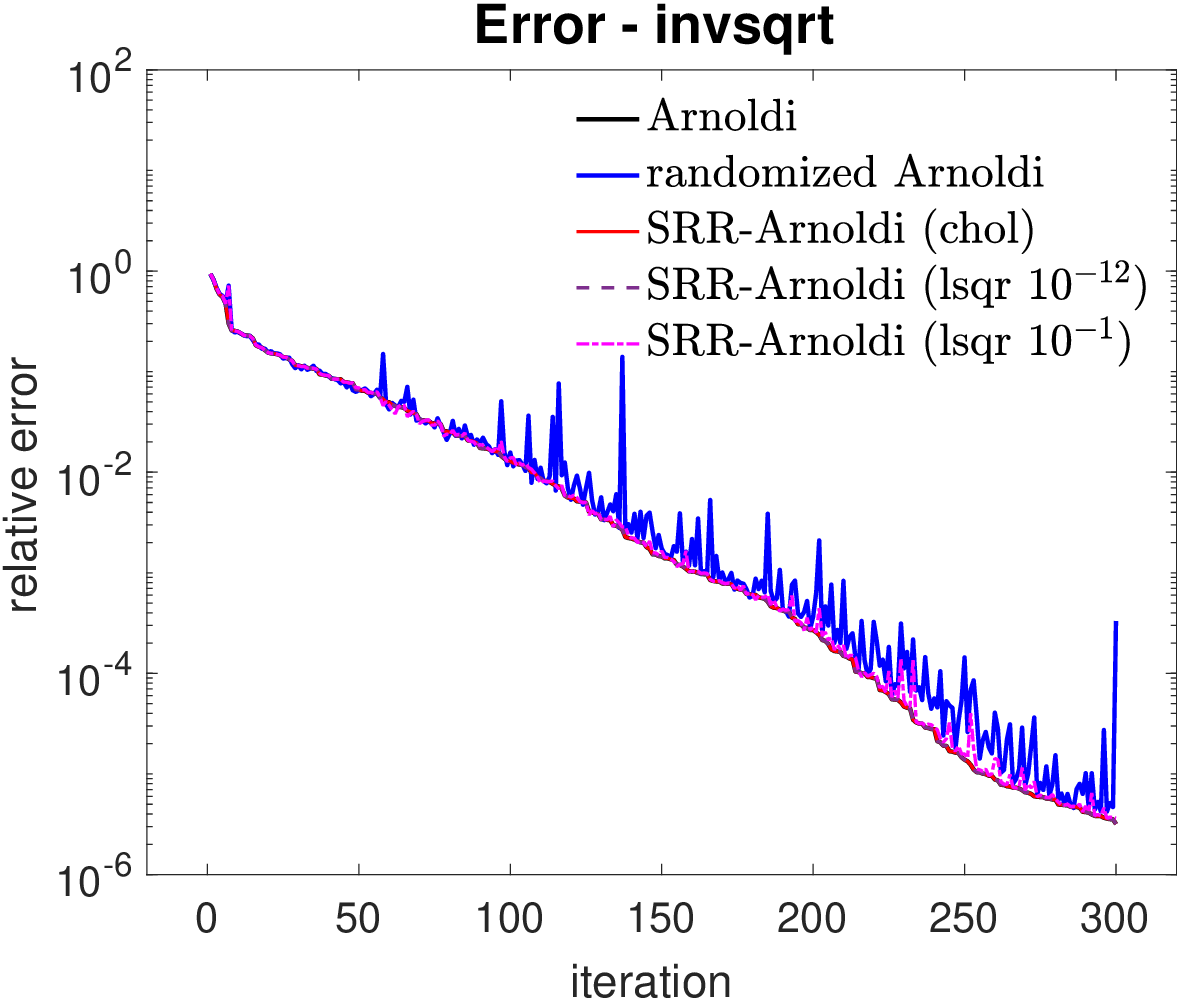}
    \includegraphics[width = \figsizeT]{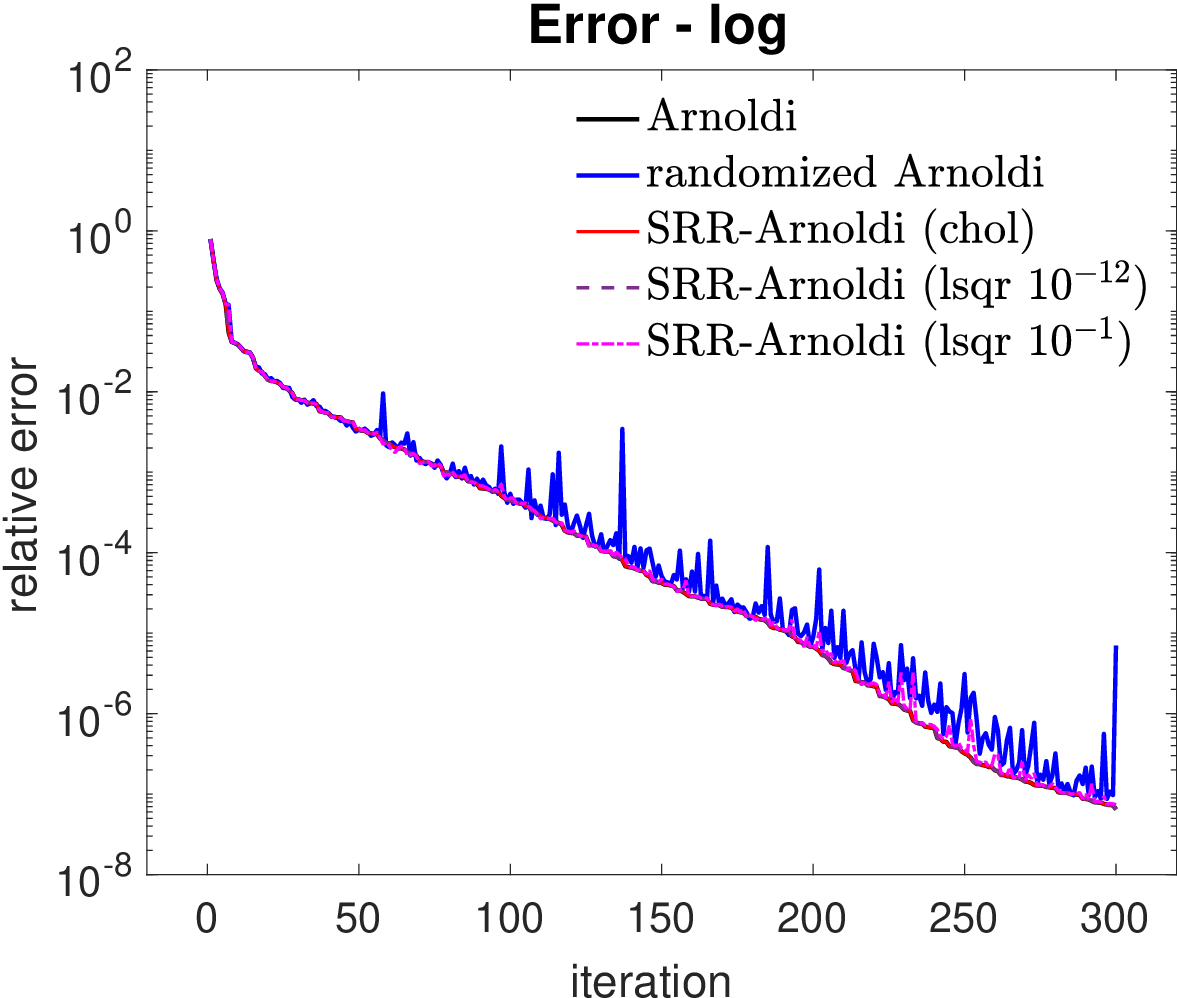}
    \includegraphics[width = \figsizeT]{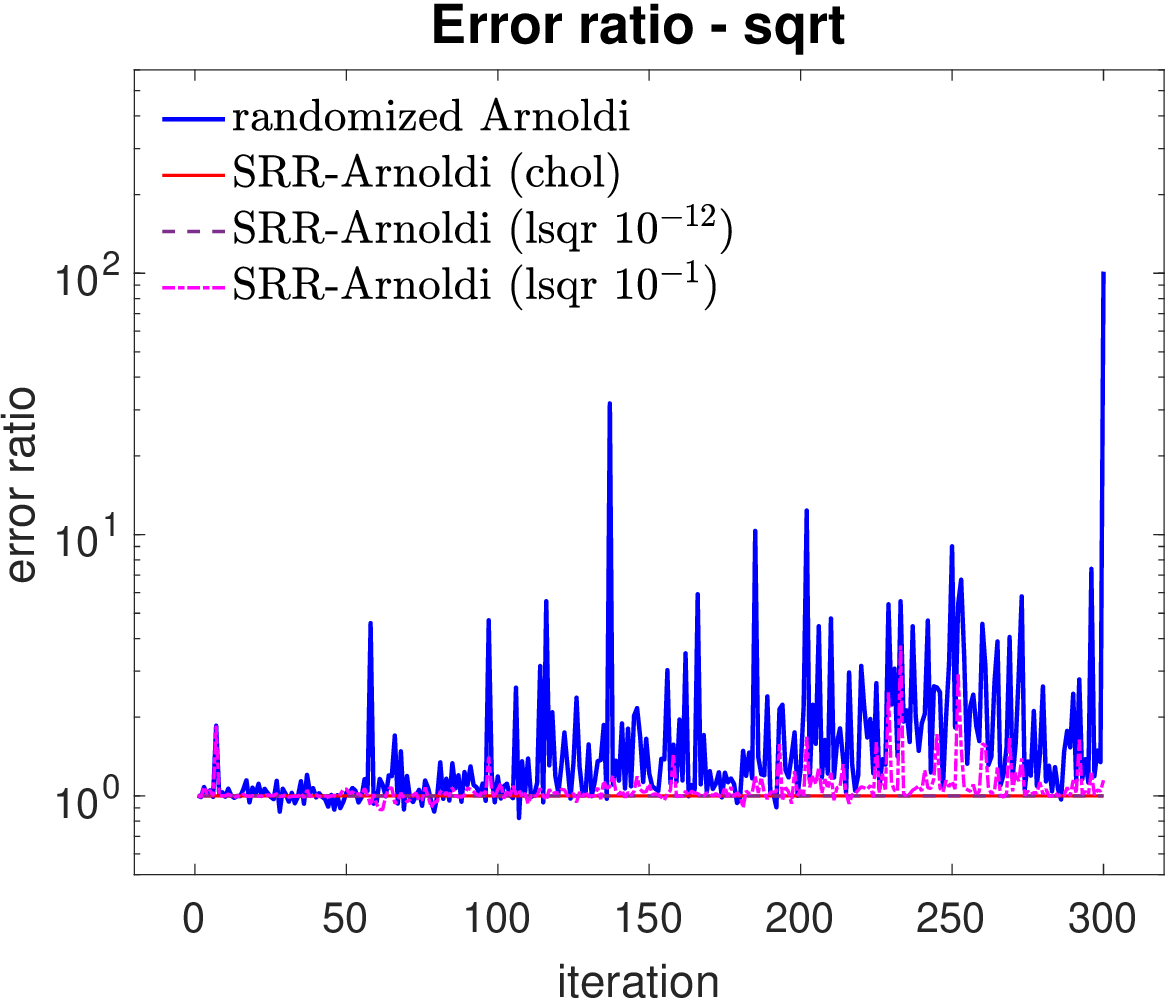}
    \includegraphics[width = \figsizeT]{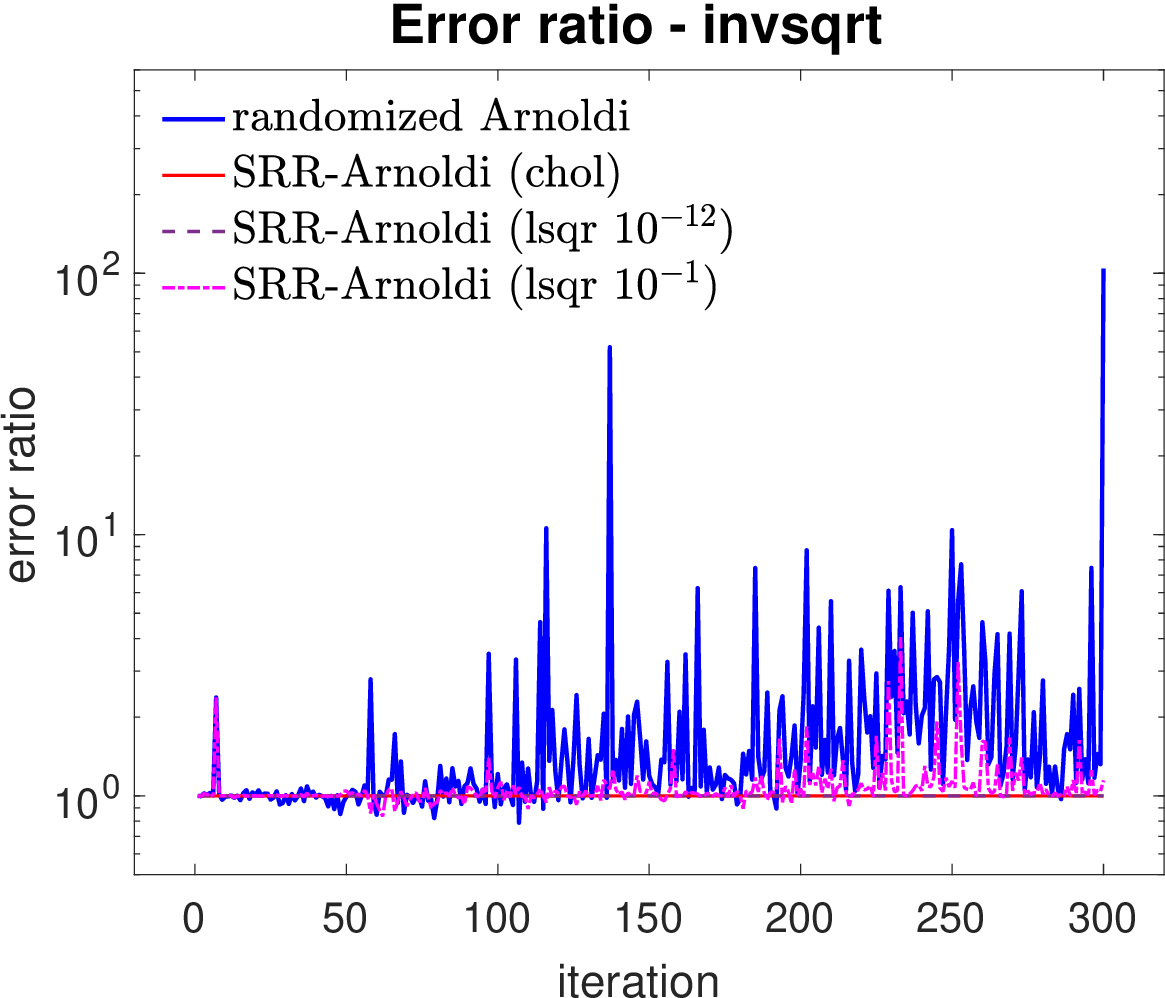}
    \includegraphics[width = \figsizeT]{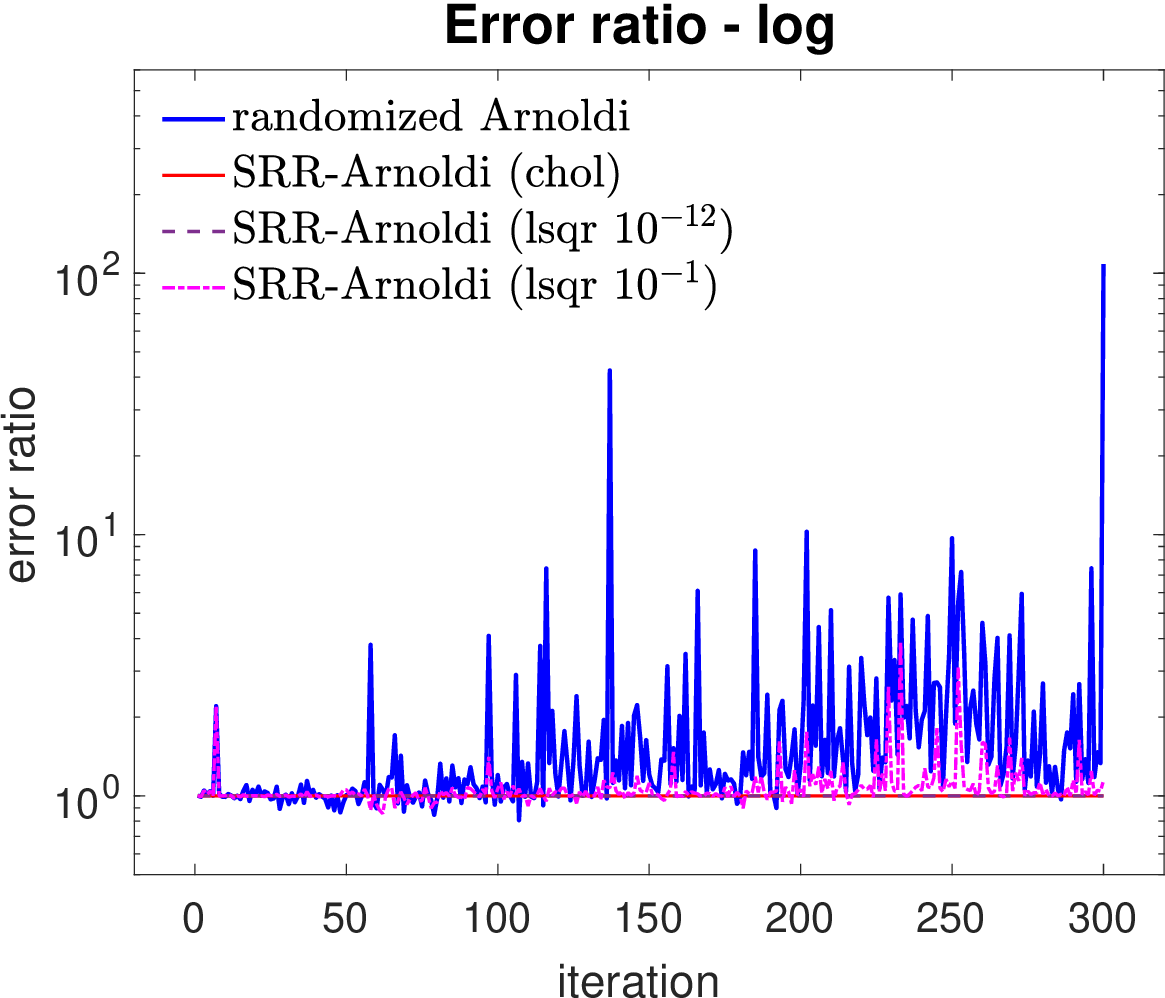}
    \caption{Convergence history for the approximation of $f(A)b$ for the test matrix in \cref{subsubsec:experiments-matfun-convergence}, comparing standard Arnoldi \cref{eqn:fAq-arnoldi-orth}, randomized Arnoldi \cref{eqn:fAq-arnoldi-rand} and SRR-Arnoldi (\cref{algo:matfun}). The curves for standard Arnoldi and SRR-Arnoldi using Cholesky or LSQR with tolerance $10^{-12}$ are completely overlapping. Top: relative errors. Bottom: ratios with respect to the error of the standard Arnoldi approximation.
	\label{fig:fAb-comparison}}
\end{figure}

\subsubsection{Execution time}
\label{subsubsec:experiments-matfun-time}

In this section, we compare the execution times of the different methods for the approximation of $f(A) b$. 
Recall that for SRR-Arnoldi, the least-squares problem~\cref{eqn:correction-least-squares-problem} must be solved every time the approximation to $f(A)b$ is computed. However, evaluating this approximation only once in the final iteration would not be realistic as it requires knowing in advance the number of iterations needed to reach a target accuracy. Therefore, to ensure a fair comparison between the different approaches, we compute the approximation to $f(A) b$ every $10$ iterations with all methods. We note that, on the other hand, the timings in \cite{CKN24} only take into account the cost of the least-squares solve for the computation of the final solution, which is only realistic if the number of iterations required to reach the target accuracy is known in advance.

The setup for this experiment is similar to \cref{subsubsec:experiments-matfun-convergence}, with the main difference that we consider a non-symmetric matrix: we define $A = Q^{\Ttran} B Q \in \R^{n \times n}$, where $Q$ is the discrete cosine transform matrix and $B$ is an upper bidiagonal matrix whose diagonal entries $B_{ii}$ are defined as follows: for $1 \le k \le 4$, the entries $B_{ii}$ with $1 + (k-1)n/4\le i \le kn/4$ are given by independent Gaussian random variables with distribution $\mathcal{N}(10^{k-1}, 10^{k-3})$. For matrices of size $\{10000, 20000, 40000, 80000\}$, all the entries on the first upper diagonal of $B$ are independent Gaussian random variables with mean zero and variances $\{ 0.37^2, 0.32^2, 0.30^2, 0.29^2 \}$, respectively. These parameters have been tuned to ensure that the number of iterations needed to reach the target accuracy $10^{-6}$ is similar across all tests.
As in \cref{subsubsec:experiments-matfun-convergence}, we compare the Arnoldi approximation \cref{eqn:fAq-arnoldi-orth} with the randomized Arnoldi approximation \cref{eqn:fAq-arnoldi-rand} and three variants of SRR-Arnoldi, using Cholesky and LSQR with tolerances $10^{-12}$ and $10^{-1}$ to solve the least-squares problem \cref{eqn:correction-least-squares-problem}.  
We set $d = 1500$, and stop each method when the relative error of the approximate solution is below a tolerance of $10^{-6}$ with respect to a reference solution calculated by running the standard Arnoldi algorithm \cref{eqn:fAq-arnoldi-orth} until it reached a relative accuracy of $10^{-8}$. The execution times and iteration counts are reported in \cref{fig:fAb-ns-times-iters}. We observe that randomized Arnoldi sometimes requires more iterations to converge compared to the other methods (most likely due to the spikes observed in \cref{subsubsec:experiments-matfun-convergence}), and that SRR-Arnoldi with Cholesky is usually only slightly more expensive than~\cref{eqn:fAq-arnoldi-rand}. 
When the least-squares problem \cref{eqn:correction-least-squares-problem} in SRR-Arnoldi is solved with LSQR, a tolerance of $10^{-12}$ significantly increases its runtime with respect to solving \cref{eqn:correction-least-squares-problem} with a Cholesky factorization, while a loose tolerance of $10^{-1}$ slightly increases the number of iterations but often slightly reduces the runtime.

\begin{figure}[t]
    \centering
    \includegraphics[width = \figsizeT]{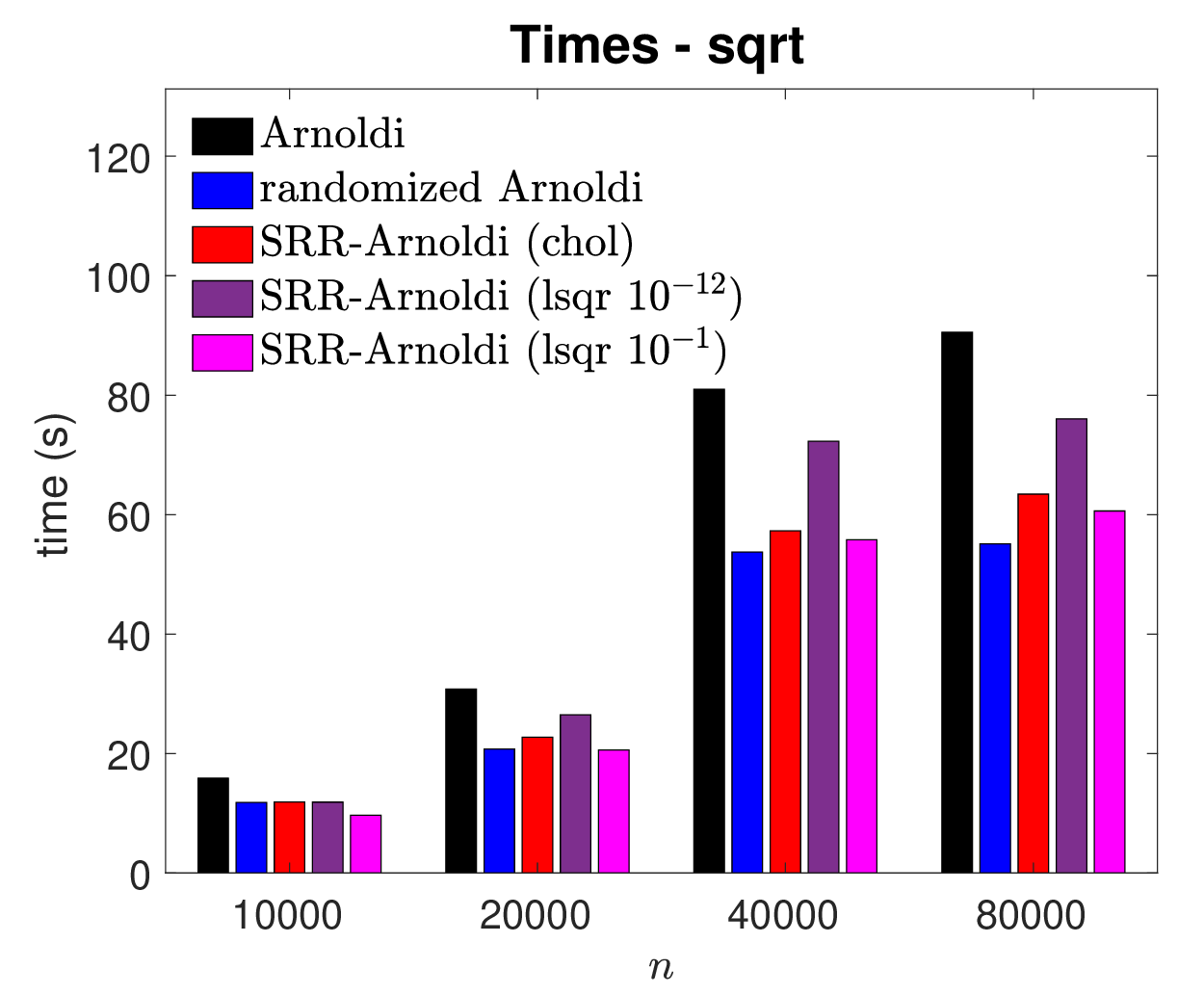}
    \includegraphics[width = \figsizeT]{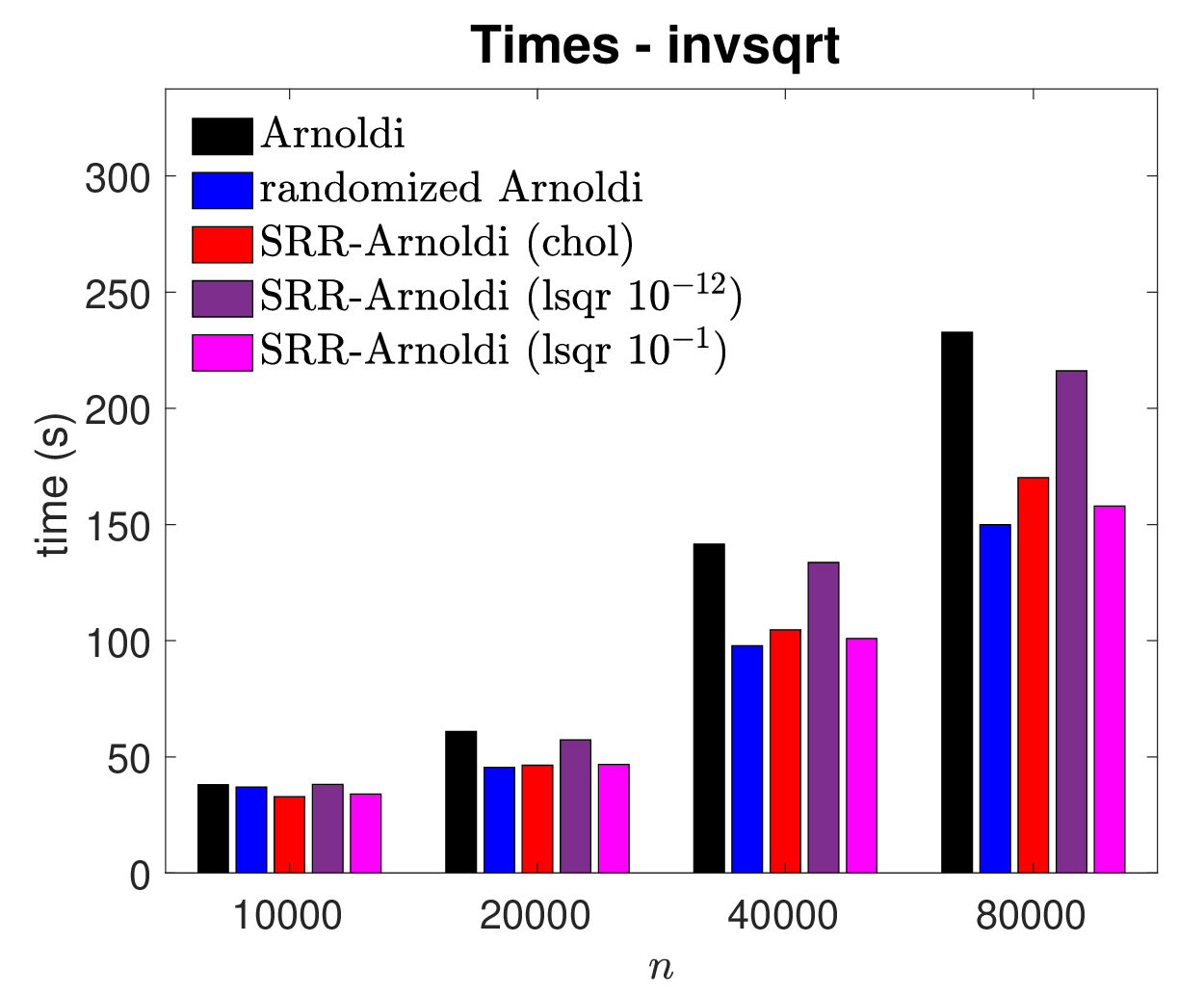}
    \includegraphics[width = \figsizeT]{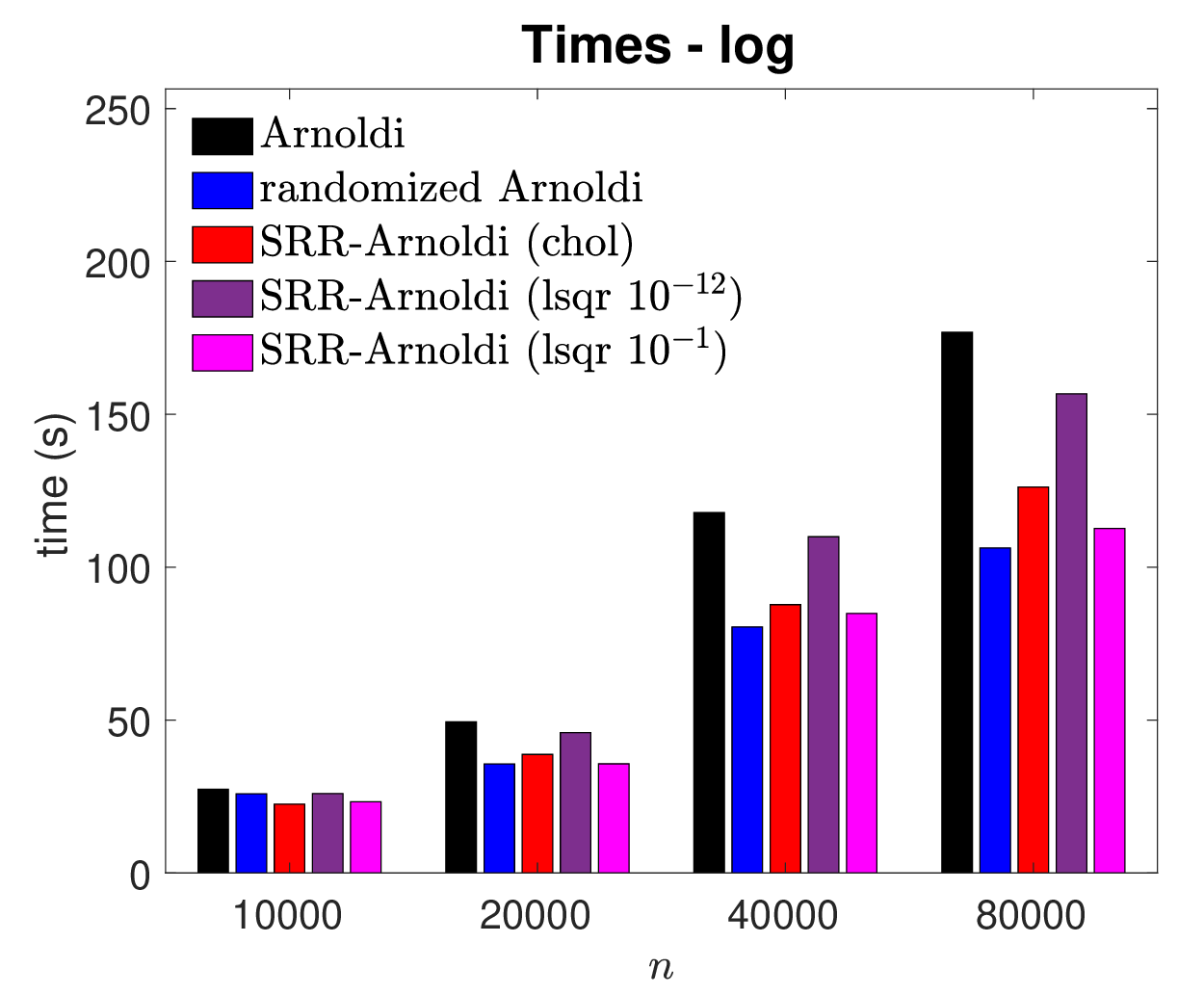}
	\includegraphics[width = \figsizeT]{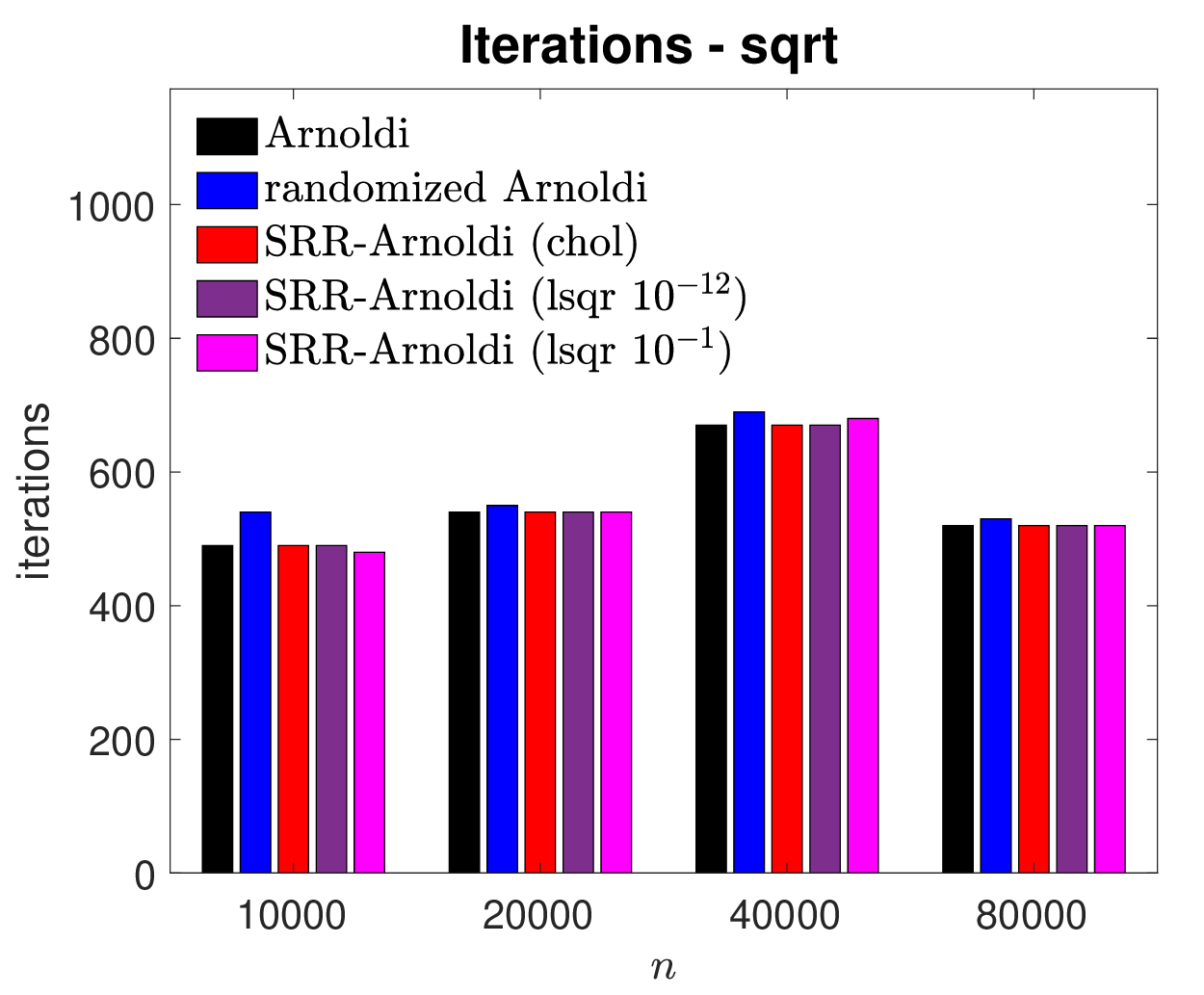}
    \includegraphics[width = \figsizeT]{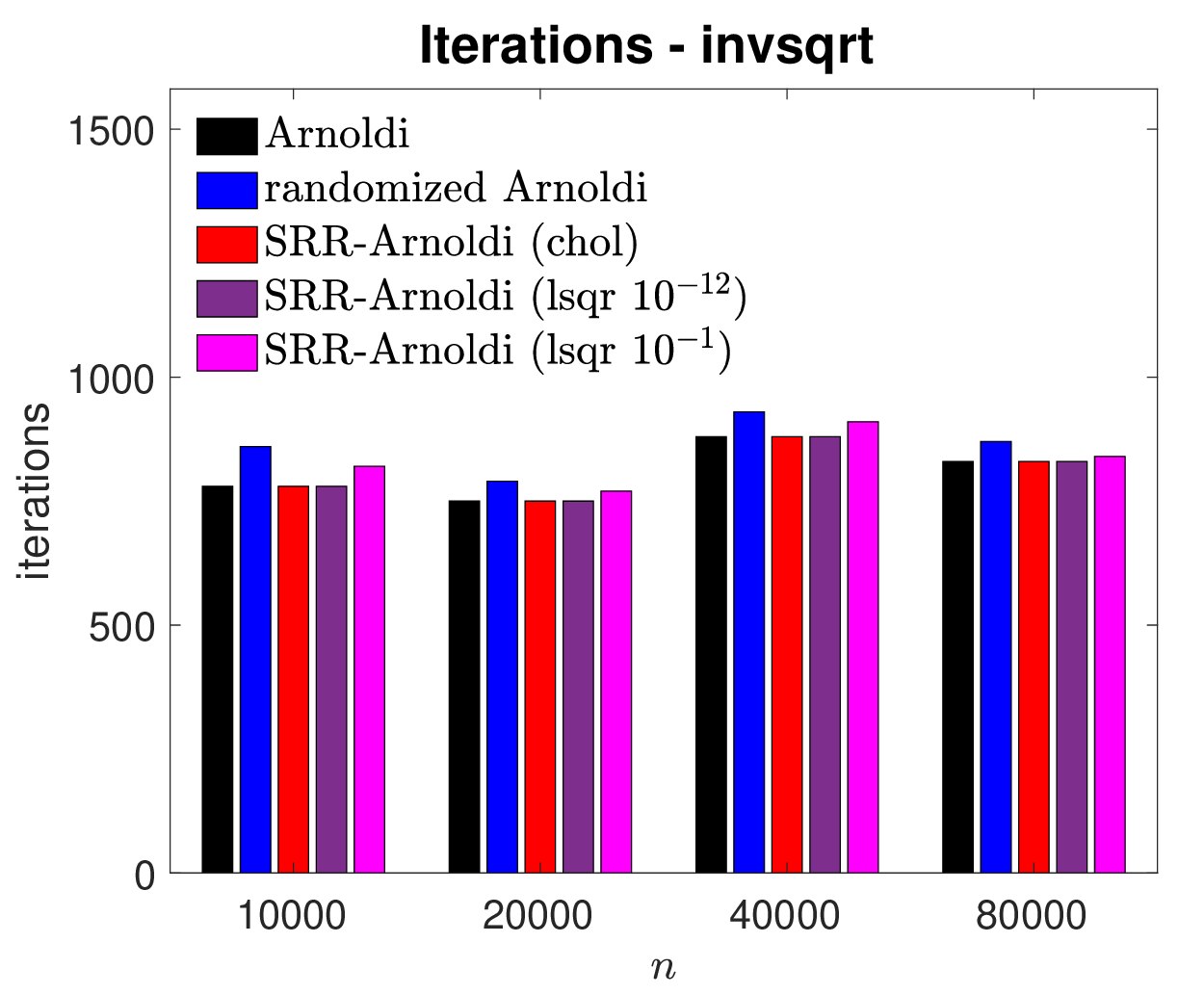}
    \includegraphics[width = \figsizeT]{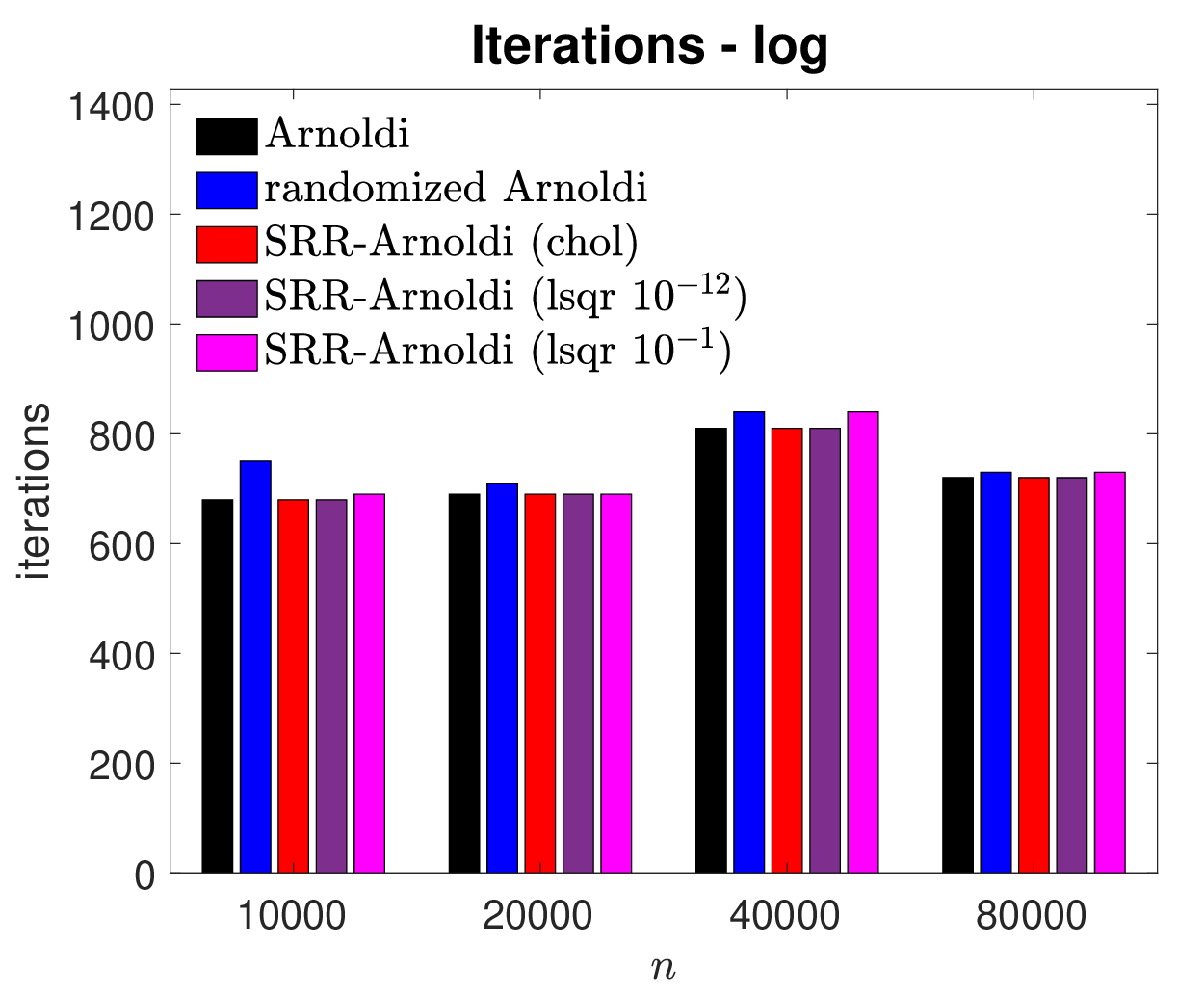}
    \caption{Times and iterations for the approximation of $f(A)b$ with synthetic non-symmetric $A$ in \cref{subsubsec:experiments-matfun-time}, comparing standard Arnoldi \cref{eqn:fAq-arnoldi-orth}, randomized Arnoldi \cref{eqn:fAq-arnoldi-rand} and SRR-Arnoldi (\cref{algo:matfun}).
	\label{fig:fAb-ns-times-iters}}
\end{figure}

\subsubsection{Square root of a graph Laplacian}
\label{subsubsec:experiments-matfun-graphlap}

In this section, we consider the Laplacians of three directed graphs and compute the action of their square root on a normalized random vector $b$; we refer to \cite[Sec.~6.2]{BenziBoito20} for details on fractional powers of a graph Laplacian and some of their applications. We use the unweighted graphs associated with the nonsymmetric matrices \texttt{big}, \texttt{foldoc}, and \texttt{p2p-Gnutella30} available from the Sparse Matrix Collection, whose sizes are, respectively, $n = 13209$, $13356$ and $36682$. We plot the convergence curves for the computation of $f(A) b$ and the times for the different methods in \cref{fig:fAb-graphlap-err,fig:fAb-graphlap-times}, setting $d = 1000$ and stopping when the relative accuracy of the approximate solution is below $10^{-6}$, where the reference solution is computed by running the standard Arnoldi algorithm \cref{eqn:fAq-arnoldi-orth} until it reaches a relative accuracy of $10^{-8}$. We observe a similar behavior to the synthetic examples in \cref{subsubsec:experiments-matfun-time}. For the graph \texttt{foldoc}, SRR-Arnoldi with Cholesky or LSQR with tolerance $10^{-1}$ is even slightly faster than randomized Arnoldi.

\begin{figure}[t]
    \centering
    \includegraphics[width = \figsizeT]{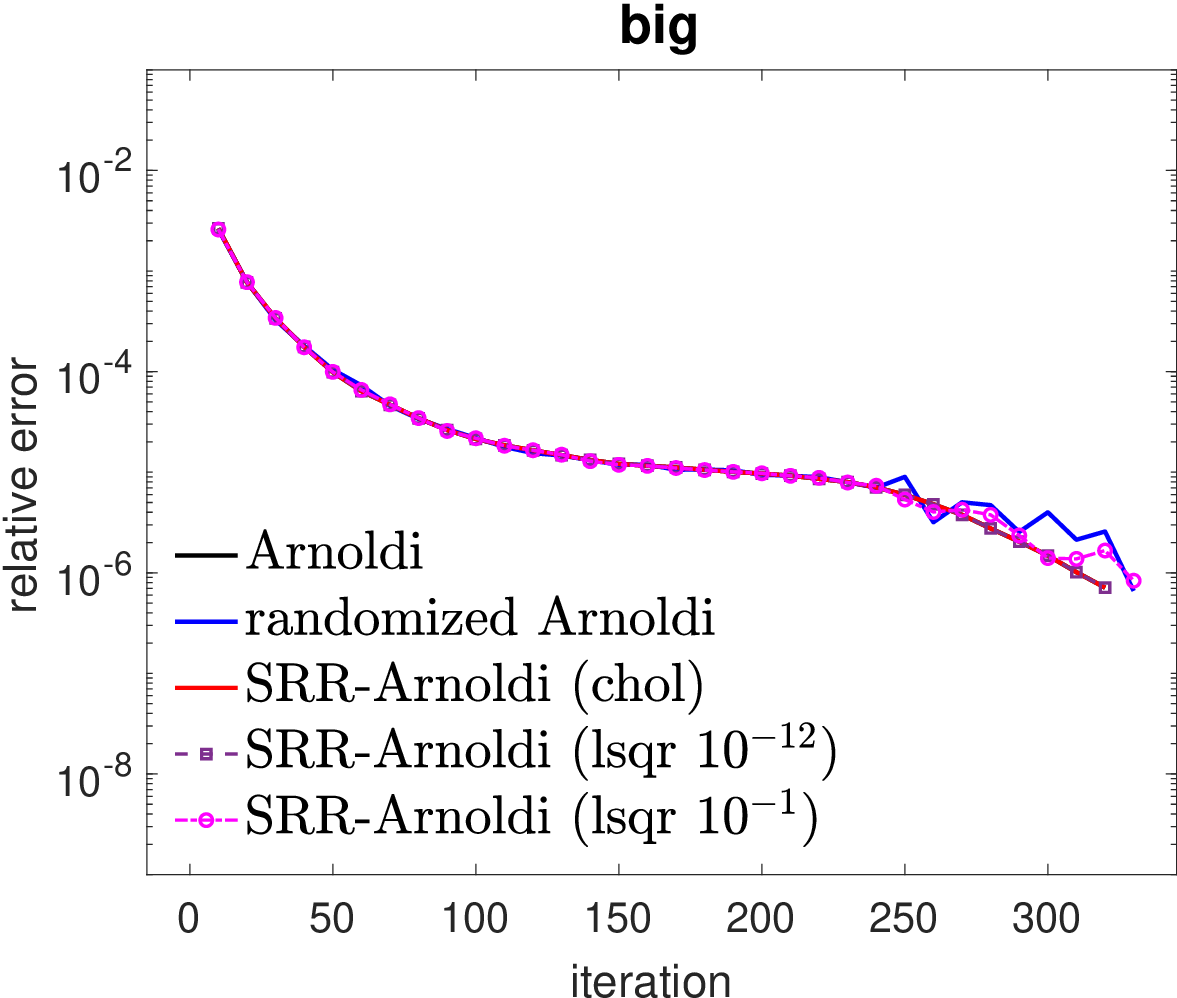}
    \includegraphics[width = \figsizeT]{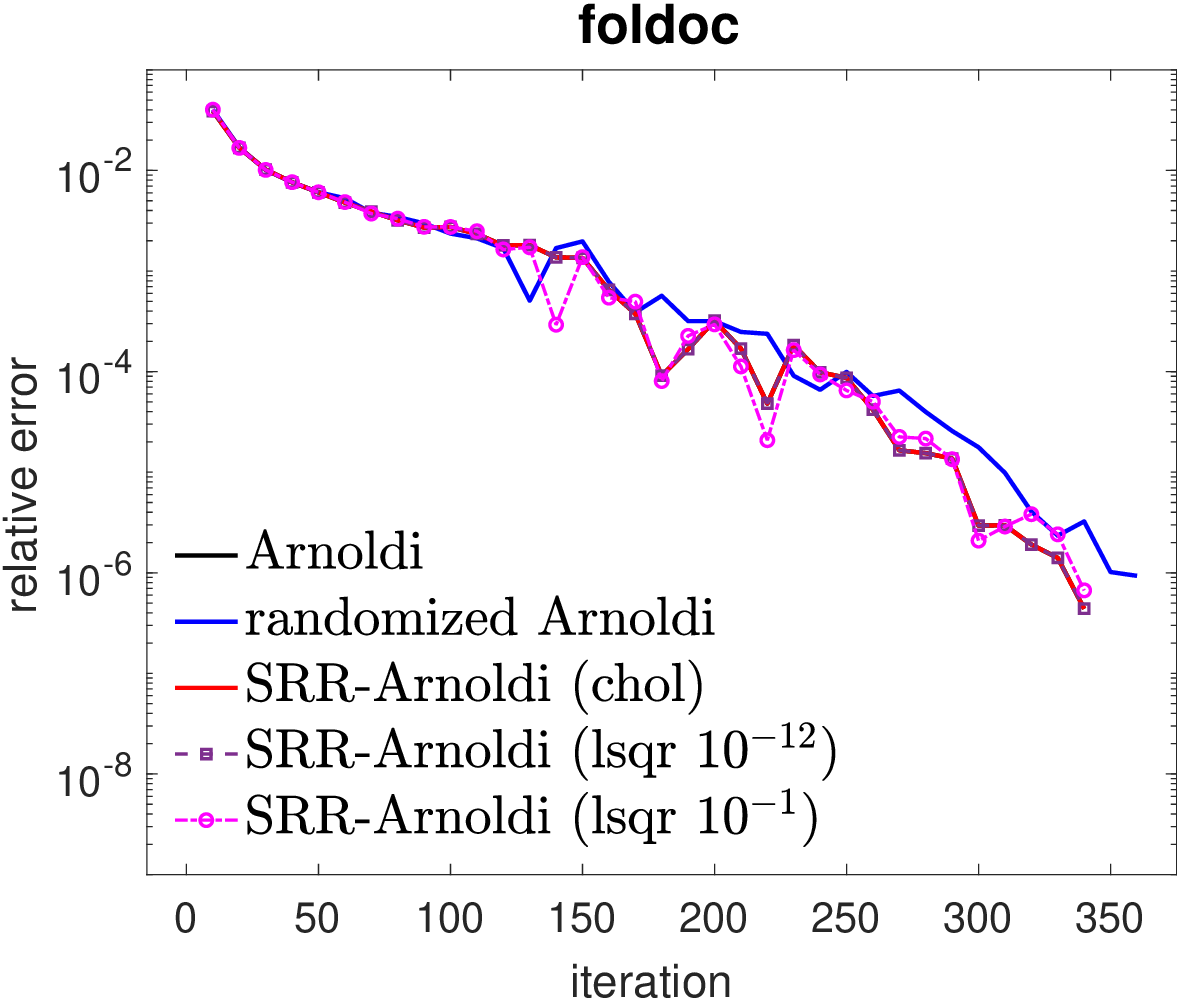}
    \includegraphics[width = \figsizeT]{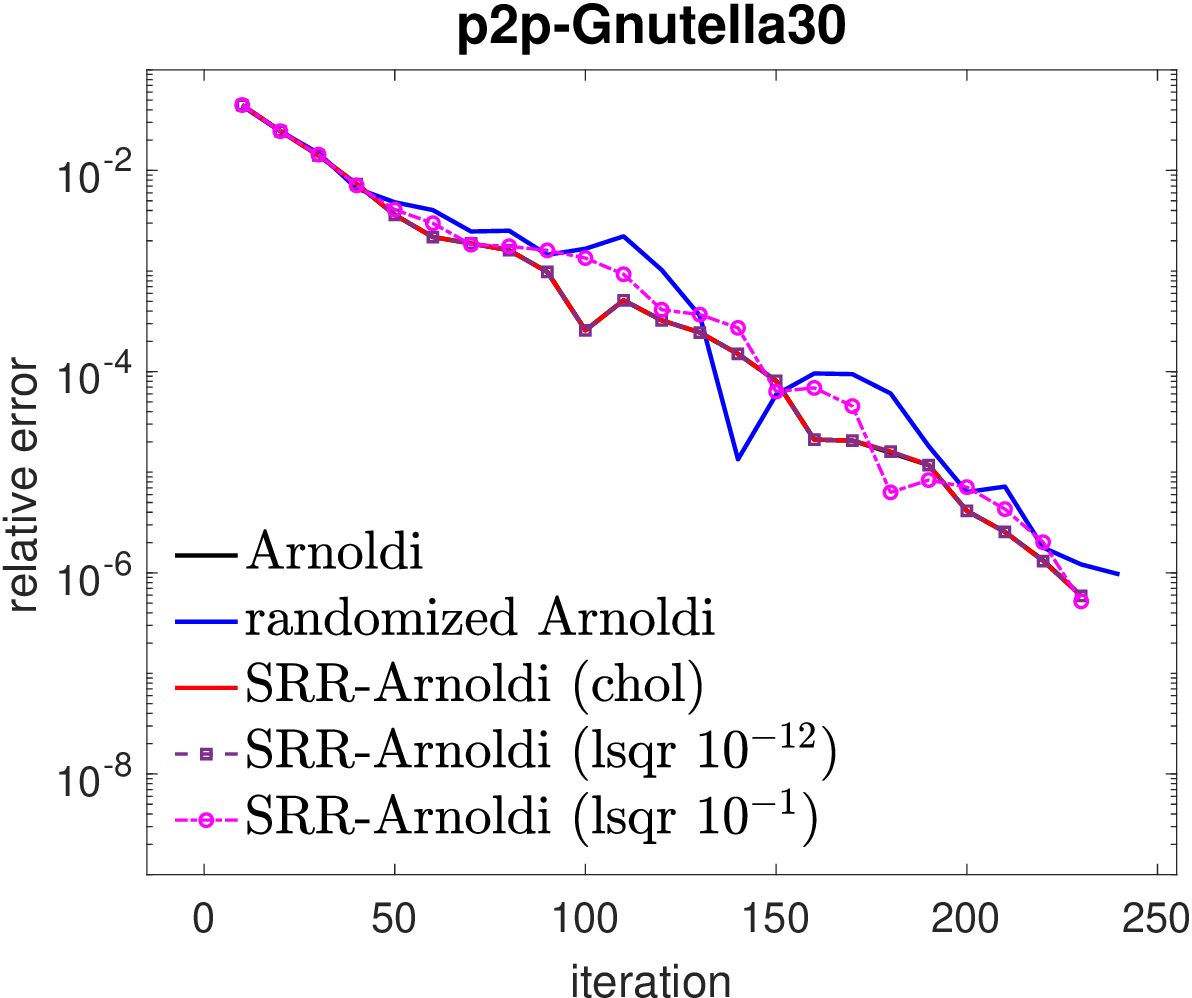}
	\caption{Errors for the computation of $A^{1/2}b$ for the Laplacians of the graphs \texttt{big}, \texttt{foldoc} and \texttt{p2p-Gnutella30} with standard Arnoldi \cref{eqn:fAq-arnoldi-orth}, randomized Arnoldi \cref{eqn:fAq-arnoldi-rand} and SRR-Arnoldi (\cref{algo:matfun}). The curves for standard Arnoldi and SRR-Arnoldi using Cholesky or LSQR with tolerance $10^{-12}$ are completely overlapping.
	\label{fig:fAb-graphlap-err}}
\end{figure}

\begin{figure}[t]
    \centering
    \includegraphics[width = 0.45\textwidth]{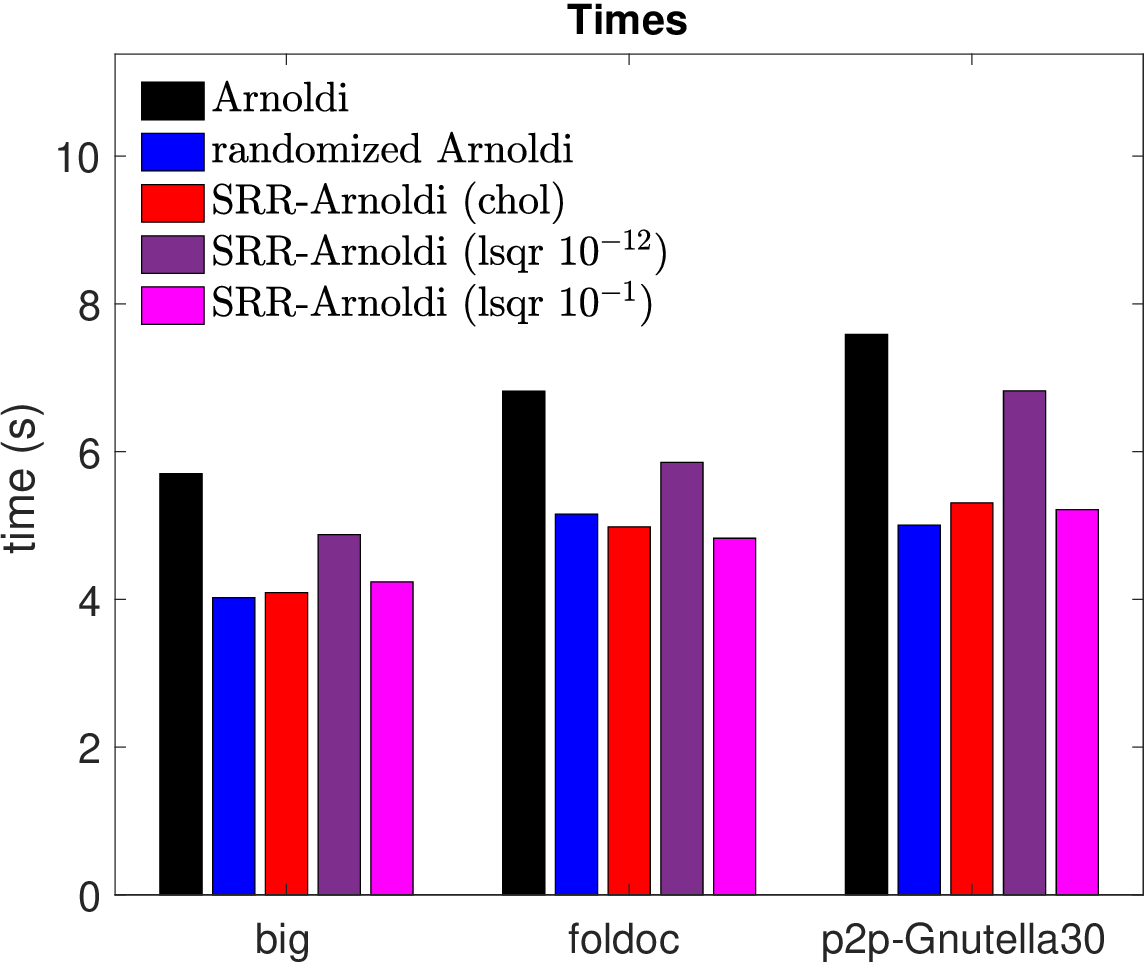}
    \caption{Times for the approximation of $A^{1/2}b$ for the Laplacians of the graphs \texttt{big}, \texttt{foldoc} and \texttt{p2p-Gnutella30} with standard Arnoldi \cref{eqn:fAq-arnoldi-orth}, randomized Arnoldi \cref{eqn:fAq-arnoldi-rand} and SRR-Arnoldi (\cref{algo:matfun}). 
	\label{fig:fAb-graphlap-times}}
\end{figure}

\section{Conclusions}
In this paper, we proposed a modification of the randomized Arnoldi process that restores similarity with the Hessenberg matrix from the standard Arnoldi method by enforcing orthogonality between the last Arnoldi vector and the previous subspace. 
Our analysis and experiments show that the modified process recovers the accuracy and robustness of the standard Arnoldi process while preserving the efficiency of the randomized variant.

\bibliography{ref}
\bibliographystyle{siamplain}
\end{document}